\documentclass[12pt,reqno]{amsart}

\usepackage{amsmath,amsthm,amssymb,comment,fullpage}

\usepackage{times}
\usepackage[T1]{fontenc}
\usepackage{mathrsfs}
\usepackage{latexsym}
\usepackage[dvips]{graphics}
\usepackage{epsfig}
\usepackage{amsmath,amsfonts,amsthm,amssymb,amscd}
\input amssym.def
\input amssym.tex
\usepackage{color}
\usepackage{hyperref}
\usepackage{url}

\newcommand{\burl}[1]{\textcolor{blue}{\url{#1}}}

\newcommand{\ch}{{\bf 1}}

\newcommand{\cal}{\mathcal}

\numberwithin{equation}{section}

\newtheorem{thm}{Theorem}[section]

\newtheorem{cor}[thm]{Corollary}

\newtheorem{prop}[thm]{Proposition}

\newtheorem{defi}[thm]{Definition}

\theoremstyle{plain}

\newtheorem{example}[thm]{Example}
\newtheorem{lemma}[thm]{Lemma}

\newtheorem{theorem}[thm]{Theorem}

\newcommand\be{\begin{equation}}
\newcommand\ee{\end{equation}}
\newcommand\bea{\begin{eqnarray}}
\newcommand\eea{\end{eqnarray}}
\newcommand\bi{\begin{itemize}}
\newcommand\ei{\end{itemize}}
\newcommand\ben{\begin{enumerate}}
\newcommand\een{\end{enumerate}}
\newcommand\bc{\begin{center}}
\newcommand\ec{\end{center}}
\newcommand\ba{\begin{array}}
\newcommand\ea{\end{array}}

\newcommand\usb[2]{\underset{#1}{\underbrace{#2}}}

\newcommand{\umessclean}[2]{\underset{=#1}{\underbrace{#2}}}

\newcommand{\C}{\ensuremath{\mathbb{C}}}
\newcommand{\Z}{\ensuremath{\mathbb{Z}}}

\newcommand{\N}{\mathbb{N}}

\newcommand\frakfamily{\usefont{U}{yfrak}{m}{n}}
\DeclareTextFontCommand{\textfrak}{\frakfamily}
\newcommand\G{\textfrak{G}}

\newcommand{\twocase}[5]{#1 \begin{cases} #2 & \text{{\rm #3}}\\ #4
&\text{{\rm #5}} \end{cases}   }

\newcommand{\hr}[1]{\href{#1}{\url{#1}}}

\title{A Probabilistic Approach to Generalized Zeckendorf Decompositions}

\author{Iddo Ben-Ari}
\email{\textcolor{blue}{\href{mailto:iddo.ben-ari@uconn.edu}{iddo.ben-ari@uconn.edu}}}
\address{Department of Mathematics, University of Connecticut, Storrs, CT 06269}

\author{Steven J. Miller}
\email{\textcolor{blue}{\href{mailto:sjm1@williams.edu}{sjm1@williams.edu}},  \textcolor{blue}{\href{Steven.Miller.MC.96@aya.yale.edu}{Steven.Miller.MC.96@aya.yale.edu}}}
\address{Department of Mathematics and Statistics, Williams College, Williamstown, MA 01267}

\thanks{The first named author was partially supported by NSA grant H98230-12-1-0225 and by grant  282912 from the Simons Foundation.  The second named author was partially supported by NSF grants DMS1265673 and DMS1561945. We thank the participants of the 2011, 2012 and 2013 SMALL REU at Williams College for many useful discussions, and the referee for helpful comments on an earlier draft, especially on related work.}



\date{\today}

\begin{document}

\maketitle

\begin{abstract}
Generalized Zeckendorf decompositions are expansions of integers as sums of elements of  solutions to recurrence relations. The simplest cases are base-$b$ expansions, and the standard Zeckendorf decomposition uses the Fibonacci sequence. The expansions are finite sequences of nonnegative integer coefficients (satisfying certain technical conditions to guarantee uniqueness of the decomposition) and which can be viewed as analogs of sequences of variable-length words made from some fixed alphabet.  In this paper we present a new approach and construction for  uniform measures on expansions, identifying them as the distribution of a Markov chain conditioned not to hit a set. This gives  a unified approach that allows us to easily recover results on the expansions from analogous results for Markov chains, and in this paper we focus on laws of large numbers, central limit theorems for sums of digits,  and statements on gaps (zeros) in expansions. We expect  the approach to  prove useful in other similar contexts. \end{abstract}





\section{Introduction}
\subsection{Background}
\label{sec:intro_intro}
A representation of the set of integers in terms of a sequence of digits is known in the literature as a numeration system. The most common numeration systems are decimal (aka radix) expansions, yet many other numeration systems appear in theory and applications, and the study of numeration systems has been an active research area in mathematics and theoretical computer science. Many of these arise from a greedy algorithm (see for example \cite{ref4}), though there are systems arising from recurrence relations where the greedy algorithm fails a positive percentage of the time (see \cite{CFHMN2,CFHMNPX}). While our focus will be on recurrence relations and greedy algorithms, other choices are possible and often closely related. These include starting from a rational language and, using an ordering inherited from an ordering of the digits, representing $n$ as the $n$\textsuperscript{th} element of the language (see \cite{ref6}), or (see \cite{ref1, ref2}) starting with a substitution $\sigma$ on a finite alphabet and encoding $n$ by the $n$ letter prefix of a fixed point of $\sigma$ (represented by concatenating iterates of $\sigma$ applied to certain letters, which are the digits), or having variable rules for which summands are available at which points in a decomposition (see the $f$-decompositions of \cite{DDKMMU}). \\

As many closely related systems are studied in different disciplines, often the same result is proved again and again, though from different vantages. Stolarsky \cite{Sto} (see also \cite{CHZ}) wrote: \emph{Whatever its mathematical virtues, the literature on sums of digital sums reflects a lack of communication between researchers.} We agree, and in fact this lack of communication was the impetus for the present paper. While many of our results are already known, we adopt a perspective used fruitfully in related problems and give a unified treatment using Markov methods (see for example \cite{ref3, ref5, ref7, MW1}) of many results previously done through combinatorial approaches. In particular, we apply these techniques to some problems that appear not to have been studied by other researchers using these methods, such as properties of gaps between summands.\\

We focus on the case where the numeration system is obtained from the greedy algorithm. Unfortunately there are several different notational conventions in the subject, depending on the perspective one adopts. We use a simple one below to motivate the problem, and  discuss the small changes later.

Fix a sequence  of integers $1=u_0<u_1< \cdots$ (also known as the basis). Then any $N\in\N$ can be represented uniquely as a combination of elements from the sequence as follows.  Let $u_n$ be the largest element in the sequence  which is $\le N$, and set $d_n =\lfloor  N / u_n\rfloor $. Continue inductively by letting $d_{k-1} = \lfloor  (N - \sum_{n \ge k\ge j  } d_j u_j)/u_{k-1} \rfloor$, for $k=n,\dots, 1$.  Clearly, the digits  $d_1,\dots,d_N$ are uniquely determined, and it is easy to see that $N= \sum_{0 \le j \le n} d_j u_j$.   We refer the reader to \cite{ref4} for more details and results.  The sequence of digits $d_n\dots d_1$, is the  word representing $N$ relative to the basis $(u_n)$. A numeration system is called regular if it can be given as an output of a finite automaton, or, equivalently, the set of words is a regular language.  It is known that  for the greedy algorithm to be regular, $(u_n)$ must satisfy a linear recurrence relation with integer coefficients \cite{shallit}. A partial converse also holds \cite{hollander}. As a result, the numeration systems associated to linear recurrence are of outmost importance for theory and applications.  The simplest examples are when  $u_n = b^n$ for some integer $b\ge 2$,  and the resulting numeration system is the base-$b$ decimal system (or $b$-radix system). The corresponding language is simply set of all word from the alphabet $\{0,\dots,b-1\}$.  When $u_1=1$, $u_2=2$ and for $n\ge 1$ we take $u_{n+1} = u_n +u_{n-1}$, we obtain the Fibonacci numeration system, also commonly and henceforth referred to as the Zeckendorf decomposition.  In this system each natural number is uniquely expressed as a sum of non-adjacent elements of the Fibonacci sequences (for us the Fibonacci sequence starts  $1,2,3,5,8,\dots$, as otherwise we do not have unique decompositions), and the corresponding language is all binary sequences starting with $1$ and  with no adjacent $1$'s, formally expressed as $1\{0,01\}^*$ where $*$ is the Kleene star. For example for $N= 11=8+3=F_5+F_3$, so that $d_5=1,d_4=0,d_3=1,d_2=d_1=0$,  and the decomposition could be viewed as the binary sequence $10100$.



\subsection{The Generalized Zeckendorf Decomposition}
\label{sec:zeckendorf}

We now introduce the generalized Zeckendorf decomposition and present some related results. This discussion is mostly a motivation and preparation for our probabilistic construction. These results have been extensively studied in the past both for the Zeckendorf and generalized Zeckendorf and also for other numeration systems, and we will discuss this in Section \ref{sec:prob_app} below.  \\

Recall that  if we define the Fibonacci numbers $\{F_n\}$ by $F_1 = 1$, $F_2 = 2$ and $F_{n+2} = F_{n+1} + F_n$, then every integer can be written uniquely as a sum of non-adjacent Fibonacci numbers. This is known as Zeckendorf's Theorem \cite{Ze}. For integers $m \in [F_n, F_{n+1})$, using a continued fraction approach Lekkerkerker \cite{Lek} proved that the average number of summands is $n/(\varphi^2 + 1)$, with $\varphi = \frac{1+\sqrt{5}}2$ the golden mean. The precise probabilistic meaning of ``average" is the expectation with respect to the uniform measure on the decompositions of integers in $[F_n,F_{n+1})$, and then Zeckendorf's theorem provides an asymptotic statement on a certain statistic under the sequence of uniform probability measures on decompositions of length $n$, as $n\to\infty$. Analogues hold for more general recurrences, such as linear recurrences with non-negative coefficients \cite{Al,BCCSW,Day,GT,Ha,Ho,Ke,Len,MW1,MW2}, generalizations where additionally the summands are allowed to be signed \cite{DDKMU,MW1}, and $f$-decompositions (given a function $f:\N \to \N$, if $a_n$ is in the decomposition then we do not have $a_{n-1}, \dots, a_{n-f(n)}$ in the decomposition) \cite{DDKMMU}. The notion of a legal decomposition below generalizes the non-adjacency condition.

\begin{defi}
\label{def:lin_recur}
Given  a \textbf{length} $L\in \N$ and  \textbf{coefficients}  $c_1,\dots,c_L\in \Z_+$ with $c_1 c_L>0$, the corresponding  \textbf{positive linear recursion} is a sequence $1=G_1,G_2,\ldots\in \N$ satisfying\begin{align}
  G_{n+1} &\ = \  c_1 G_n + c_2 G_{n-1} + \cdots + c_n G_{1}+1,~n=1,\dots, L-1,\nonumber\\
   G_{n+1} &\ = \  \sum_{j=1}^L c_j G_{n+1 - j},~ n=L,L+1,\dots.
 \end{align}
\end{defi}

\begin{defi}
\label{def:legal} Given a positive linear recursion with coefficients $c_1,\dots,c_L$, an integer $N$ has a \textbf{legal} decomposition of length $n\in \N$ if there exist  $a_1\in \N, a_2,\dots,a_n \in \Z_+$, such that
\be\label{eq:decomposition}
 N\ =\  \sum_{i=1}^{n} {a_i G_{n+1-i}},\ee
and
\begin{itemize}
\item $n<L$ and $a_i=c_i$ for $1\le i \le n$; or
\item  there exists some $s\in\{1,\dots, L\}$ such that
\be \label{eq:legalcondition2}
 \left. \begin{array}{l}
  a_1\ = \ c_1,\ a_2\ = \ c_2,\ \dots,\ a_{s-1}\ = \ c_{s-1}, \mbox{ and } a_s<c_s, \\
  a_{s+1}, \dots, a_{s+\ell} =  0\mbox{  for some }\ell \ge 0,\\
  \{b_i\}_{i=1}^{n-s-\ell}\mbox{ with }b_i = a_{s+\ell+i}, \mbox{ is either legal or empty.}
\end{array}\right\} \ee
 \end{itemize}
\end{defi}

We remark that the notation above differs slightly from the representation as $\sum_j d_j u_j$; because of our  use of the recurrence relation for our analysis it is more convenient to index this way. To emphasize this we now use $a_i$ for the digits and $G_n$ for our sequence. It is important that $c_1 c_L > 0$, as when this fails there are some sequences where decompositions still exist but are no longer unique, and others where the decompositions are still unique; see \cite{CFHMN1,CFHMN2,CFHMNPX,DFFHMPP}. The following theorem has been proved many times (see for example \cite{MW1}), and is the starting point for our investigations.

\begin{thm}[Generalized Zeckendorf Decomposition]
\label{th:legal}
Consider a positive linear recurrence  with coefficients   $c_1,\dots,c_L$ and $c_1 cL > 0$. Then  every $N\in\N$ has a unique legal decomposition.
\end{thm}

The main idea in the theorem is to identify the notion of legal decomposition from \eqref{def:legal} with the representation obtained from the greedy algorithm.  The  characteristic polynomial  for the recurrence relation is given by Lemma \ref{lem:char} and is equal to $p(x) = x^L - \sum_{j=1}^L c_j  x^{L-j} $. Its Perron (aka dominant) eigenvalue $\lambda_C>1$, and satisfies $1= \sum_{j=1}^L c_j \lambda_C^{-j}$, and it then follows from \cite[Theorem 8.1]{hollander} that the generalized Zeckendorf decomposition is regular. Here is a corresponding finite automaton.  The states are labeled $(i,j)$, where $i=1,\dots,L$ and $j\in\{0,\dots,c_i\}$ for $i<L$ and $j\in\{0,\dots,c_i-1\}$ if $i=L$. If $L>1$, the initial states are $(1,0),\dots, (1,c_1)$. The transitions are as follows. From $(i,j)$ where $j<c_i$, there an edge to all states of the form $(i,j')$, and  if $j=c_i$ (only possible when $i<L$), then there an arrow to all states of the form $(i+1,j')$. As an example of how this works, consider the recurrence relation of length $L=3$ with $c_1=c_2=c_3=1$. Then we have $(G_n)_{n\in\N}= (1,2,3,6,11,20,37,\dots)$. Consider the word $1101$. Then the corresponding path for the automaton is $(1,1)\to (2,1) \to (1,0)\to (1,0)$, and it is accepted. If, however, we consider the word $1110$ then the first two  vertices in the path are  $(1,1)\to (2,1)$. However, since $c_2=1$, $L=3$ and $c_3=1$, it follows that the only allowed transition from $(2,1)$ is to $(3,0)$, but as the third digit is equal to $1$, this sequence is rejected. In fact, the accepted sequences are exactly those beginning with $1$, and having no three consecutive ones, which we can formally write as the regular language  $\{1,11\}\{0,01,011\}^*$, where $*$ is the Kleene star, and this is exactly the set of legal decompositions. \\

In the sequel we will fix a linear recurrence as in Definition \ref{def:lin_recur}. From Theorem \ref{th:legal} it follows that there's a one-to-one correspondence between the set of integers in $[G_n,G_{n+1})$
 through \eqref{eq:decomposition}, where the integer $N$ is mapped to its legal decomposition $(a_1(N),\dots,a_n(N))$. Let $Q_n$ denote the uniform distribution on the legal decompositions of integers in $[G_n,G_{n+1})$, and with this identification it is natural to consider  $N$ and $a_1(N),\dots,a_n(N)$ as random variables. In what follows, we denote expectation with respect to $Q_n$ by $E^{Q_n}$.\\

For  $N \in [G_n,G_{n+1})$, \eqref{eq:decomposition} can be rewritten as
\be\label{eq:kN_def}  N \ = \ G_{i_1(N)} + G_{i_2(N)} + \cdots + G_{i_{k(N)}},\ee
where $1\le i_1 \le \dots \le i_{k(N)}\le n$. The random variable $k(N)$ gives the number of summands, in the generalized Zeckendorf decomposition, or the sum of digits, that is, $k(N)=\sum_{i=1}^n a_i(N)$. It was the main object of previous works. The first result was Lekkerkerker's theorem on the asymptotic expectation of $k(N)$ when $G_n= F_n$. Here is its generalization to our setting.

   \begin{thm}[Generalized Lekkerkerker's Theorem]
\label{th:gen_lek}
 There exist constants $C_{{\rm Lek}} > 0$ and $d$ such that
 \be E^{Q_n} k(N)\ =\ C_{{\rm Lek}}n + d + o(1) \ \mbox{as}\ n\to\infty.\ee
\end{thm}

Many of the proofs of Theorem \ref{th:gen_lek} are plagued by the need to prove results about roots of the characteristic polynomials associated to the recurrence in order to show $C_{{\rm Lek}} > 0$; recently, though, a combinatorial approach was developed in \cite{CFHMNPX} which bypasses these technicalities.

Once the average number of summands has been determined, it is natural to investigate other and finer properties of the decompositions.  Three natural questions concern the fluctuations in the number of summands $k(N)$ about the mean, the distribution of gaps $i_{j+1}(N) - i_j(N),~j=1,\dots,k(N)-1$ between adjacent summands, and the length of the longest gap in a decomposition. For positive linear recurrences as in Theorem \ref{th:legal}, the distribution of the number of summands converges to a Gaussian with computable mean and variance, both of order $n$. There is an extensive literature on these results. See \cite{DG,FGNPT,GTNP,LT,Ste1} for an analysis using techniques from ergodic theory and number theory, and \cite{KKMW,MW1,MW2} for proofs via a combinatorial perspective.  These results hold true for other numeration systems and  are exactly the kind of results referred to by Stolarsky in the quote given in Section \ref{sec:intro_intro}. As before, all these are statements on the asymptotic behavior of certain statistics of generalized Zeckendorf decompositions of integers in $[G_n,G_{n+1})$ under the uniform measure, as $n\to\infty$. \\

Results on the distribution of gaps between adjacent summands have recently been obtained by Beckwith, Bower, Gaudet, Insoft, Li, Miller and Tosteson \cite{BBGILMT, BILMT}. They show that the distribution of gaps larger than the recurrence length converges to that of a geometric random variable whose parameter is the largest eigenvalue of the characteristic polynomial of the recurrence relation. For gaps smaller than the recurrence relation closed forms exist for special recurrences, though with enough work explicit formulas can be derived for any given relation. They also determine the distribution of the longest gap, and prove the behavior is similar to that of the length of the longest run of heads in a sequence of tosses of a possibly biased coin. Their proofs are a mix of combinatorics and a careful analysis of polynomials associated with the recurrence relations. The details become involved as some of the associated polynomials depend on the interval $[G_n, G_{n+1})$ under consideration.

\subsection{Probabilistic Approach}
\label{sec:prob_app}
Most  results mentioned in Section \ref{sec:zeckendorf}  above are not unique to the generalized Zeckendorf, and similar and even finer results were obtained for other numeration systems. A recurring subject of study is the sum of digits function, which, as in the case of generalized Zeckendorf, we denote by $k(N)$. The sum of digits has a natural generalization to additive functions, that is that instead of summing the digits, the summation is over some fixed function applied to each digit (example: the indicator that the digit is not zero, and the resulting sum is the number of nonzero digits. This is the same as $k(N)$ for the standard Zeckendorf and for the binary system). We note that in many of the works, these additive functionals are referred to as sums of digits functions or additive functions.  The recent survey paper \cite{CHZ} presents results on sum of digits for the base-$b$ expansion, under the uniform measure on $[1,\dots, N)$, and includes a very rich  list of bibliography on the topic, including other numeration systems. Two other works we would like to highlight are \cite{ref3}, which provides expressions for limiting distributions for regular languages,  based on combinatorial and  matrix analysis, and  \cite{ref7}, which studies the additive functional through analysis of a corresponding time-inhomogeneous Markov chains. \\

So why another work on this topic? We believe that we have a new approach, which allows for a more comprehensive treatment, and is not limited to additive functionals. Specifically, what we provide here is a tractable analytic expression for the uniform distribution on generalized Zeckendorf decompositions of fixed length, that is for random numbers in the intervals of the form $[G_n,G_{n+1})$. The reason why we focus on these intervals is because this is were the structure has the simplest  expression (though with additional work the results can be extended to $[1, N)$, as shown in Appendix C of \cite{BILMT} and \S\ref{sec:geninitialseg}). The reason why we chose the generalized Zeckendorf is because of the large body of work on the generalized Zeckendorf in the setting of fixed-length decompositions, mentioned above, which  was the motivation for the present work, a natural setting to our construction  and a  reference point to examine our new approach to the model. \\

The main  idea concerns the problem of constructing uniform measures on words of fixed length $n$   from some alphabet under certain prescribed constraints. The alphabet is the set of digits $\{0,\dots,\max c_i\}$, the word is a sequence of length $n$ from the alphabet, and the constraint is that the word yields a legal decomposition. The uniform measure we are interested in is then the uniform measure on the set of legal decompositions of length $n$. In the base-$b$ case, the alphabet is $\{0,\dots,b-1\}$ and there is no constraint, in the Zeckendorf case, the alphabet is $\{0,1\}$ and the constraint is to have no consecutive $1$'s. In the generalized Zeckendorf, we will consider a similar, yet more complex constraint.  We  construct the uniform measure on legal decompositions from the uniform measure on the sequence of digits, that is, when the digits are IID, by conditioning. The observation is  that if the constraints are in some sense shift-homogenous and localized -- which is exactly the case for the generalized Zeckendorf decomposition -- then they can be realized through a stopping rule for the IID sequence, which eventually is reduced to a hitting time of a time-homogeneous Markov chain, and our uniform measure under constraints is then viewed as a Markov chain conditioned not to hit some set. Through some elementary transformations this conditioned measure coincides with the distribution of a  time-homogeneous Markov chain known in the literature as Doob's $h$-process, pinned to a point after $n$ steps. In other words, the analysis of the uniform measure boils down to the analysis of a certain related time-homogeneous Markov chain. We note that all the quantities above depend on the length of the sequence only through the time the Markov chain is pinned, so that regardless of the length of the decomposition, we only need to consider the evolution of a single Markov chain. This identification gives a very simple expression and characterization of the uniform measure on legal decompositions, which allows to compute many quantities with little effort, as we show in later sections. Furthermore, this approach gives access to the vast literature on Markov chains, specifically asymptotic results, but not limited to, as we have a simple formula for the uniform measure in terms of the Markov chain. \\

We illustrate our method by studying the classical problems of additive functionals including mean, law of large numbers and central limit theorem, as well as obtain new results on the distribution of gaps between non-zero digits in decompositions.

\subsection{Organization}
The paper is organized as follows. In Section \ref{sec:prob} we describe the Markovian model, and how to obtain large-time asymptotics for our model from that of the underlying Markov chain. In Section \ref{sec:additive} we present  the results on additive functionals in a setting which includes our particular model, first by introducing the theoretical results in Section \ref{sec:additive_theory} and then applying them to the generalized decompositions in Section \ref{sec:application_zeckendorf}.  These results  include the classical results in this area:  sharp estimates on expectation, a  law of large numbers and a central limit theorem. In Section \ref{sec:gaps} we then apply the results on additive functionals (or the sum of digits)  to the generalized Zeckendorf decompositions.  In Section \ref{sec:gaps} we treat the gap distribution as a consequence of the regenerative structure of the underlying Markov chain, and the analogy with Bernoulli trials. Finally, in the Appendix we explain how to extend our results for decompositions of fixed length, or numbers in the interval $[G_n,G_{n+1})$ to numbers in intervals of the form $[1,N)$.
\section{Probabilistic Approach}\label{sec:prob}
We remind that throughout the discussion we assume that $L\in\N$ and  the coefficients $c_1,\dots,c_L \in \Z_+$ satisfy $c_1 c_L>0$ as in Definition \ref{def:lin_recur}. \\

The main idea is to show that for a given $n\in \Z_+$, the uniform distribution on generalized Zeckendorf decompositions consisting of $n+1$ digits (that is, the $(n+1)$-th digit is non-vanishing and all higher digits are not present) coincides with the distribution of a certain conditioned Markov chain. This provides a unified framework for the model, which, in particular, gives rather easy access to many asymptotic results. We first define the Markov chain. Let $(X,Y) = \bigl((X_n,Y_n) :  n\in\Z_+\bigr)$ be the two-dimensional process with  $X_n  \in \{0,\dots ,\max_i c_i\}$ and $Y_n \in \{1,\dots ,L\}$. The idea is that $X_0,X_1,\dots$ will be used to represent the coefficients $a_i$ in \eqref{eq:decomposition}, while $Y_0,Y_1,\dots$ will be used to keep track whether the $X_n$'s satisfy the condition \eqref{eq:legalcondition2}. This will be explained below, after we finish describing our construction. Let $P$ denote the distribution under which this is  an IID process, $(X_0,Y_0)$ being uniformly distributed over $\{0,\dots,\max_i c_i\}\times \{1,\dots,L\}$.

\begin{defi}
\label{def:legal_real}
Suppose  $L\in \N$ and  $c_1,\dots,c_L\in \Z_+,~c_1c_L>0$ are the coefficients of a linear recursion. We say that the realization $\bigl((X_0,Y_0),(X_1,Y_1), \dots\bigr)$ of the process $(X,Y)$ is \textbf{legal} with respect to the recursion if
\begin{enumerate}
\item $X_0>0$ and $Y_0=1$,
\item there exists a random variable  $J \in \Z_+$ such that $X_J>0$, $X_n =0$ and $Y_n=1$ for $n> J$,
\item For all $n\in \N$, either
\begin{enumerate}
\item $X_n <c_{Y_n}$ and $Y_{n+1}=1$ or
\item $X_n= c_{Y_n}$ and $Y_n=Y_{n+1}+1$.
\end{enumerate}
\end{enumerate}
\end{defi}
Note that condition 3b and the assumption that $Y_n \in\{1,\dots, L\}$ for all $n$ implicitly mean that in a legal realization  $X_n = c_{Y_n}$ only if $Y_n < L$. \\

The main observation is the following. Given a legal realization and letting (compare to \eqref{eq:decomposition})
\be\label{eq:decomp2} N\ =\ \sum_{j=0}^n  X_j G_{n-j+1}, \ee then $(X_0,\dots,X_n)$ is the legal decomposition of $N \in [G_{n+1},G_{n+2})$, according to Definition \ref{def:legal}.

Let
\begin{equation}\tau \ = \ \inf \{n \in \Z_+: \left ( (X_0,Y_0) , (X_1,Y_1), \dots,(X_n,Y_n)\right) \mbox{ does not extend to a legal realization}\}.\end{equation}

With a slight abuse of notation, let $Q_n$ be the probability measure on the $\sigma$-algebra generated by $(X_0,Y_0),\dots,(X_n,Y_n)$ defined through
\begin{equation} Q_n (B) \ = \ P( B | \tau >n).\end{equation}
 Since $P$ is uniform, $Q_n$ is uniform over all finite realizations $(X_0,Y_0),\dots,(X_n,Y_n)$ that extend to legal realizations. Any such finite realization corresponds to a unique Zeckendorf decomposition of length $n+1$ given in \eqref{eq:decomp2}. Conversely, every integer with Zeckendorf decomposition of length $n+1$ corresponds to a unique finite realization $(X_0,Y_0),\dots,(X_n,Y_n)$ extending to a legal realization. Therefore $Q_n$ could be identified with the uniform distribution on generalized Zeckendorf decompositions of length $n+1$. \\

We now define an auxiliary process that allows us to introduce ideas on conditioned Markov chains. The reason for doing that is the following:  $\tau$ is not a hitting or even stopping time for $(X,Y)$, as in order to determine whether $\tau=n$, it is evident from Definition \ref{def:legal_real}(3b) that on certain circumstances  the value of $Y_{n+1}$ is needed. Therefore,  the probabilistic analysis of Markov chains through stopping times, and which is key to our approach,  cannot be applied.  To fix this, let $Z_n = (X_n,Y_n,Y_{n+1})$, and let $Z=(Z_n:n\in\Z_+)$. Below we will write $Z_n (1)$ for $X_n$, $Z_n(2)$ for $Y_n$ and $Z_n(3)$ for $Y_{n+1}$.  
 It is easy to see that $\tau$ is a hitting time for $Z$. Specifically, letting
\begin{align}
{\cal L} &\ = \ \{ (x,j,j'): (x < c_j \mbox{ and }j'=1) \mbox{ or } (j<L \mbox { and } x=c_j \mbox{ and } j'=j+1) \};\nonumber\\
\label{eq:L_0} {\cal L}_0&\ =\ {\cal L} \cap \{ (x,1,j'):x>0\},
\end{align}
 then
 \begin{equation}\twocase{\tau \ = \ }{0}{if $Z_0 \not \in {\cal L}_0$}{\inf \{n: Z_n \not \in {\cal L}\}}{otherwise.}\end{equation}

Under $P$, $Z$ is a Markov chain. We abuse notation and denote its transition function by $P$ as well. Since the measure $P$ is uniform, it immediately follows that  the restriction $P_{\cal L}$ of  the transition function $P$ to ${\cal L}\times {\cal L}$ is an irreducible and aperiodic substochastic matrix.  From the Perron-Frobenius theorem we  know that $P_{\cal L}$ possesses a Perron root $\lambda_c \in (0,1)$ and corresponding left and right eigenfunctions, $\nu_c$ and $\varphi_c$, respectively, whose entries are strictly positive.  We normalize them so that $\varphi_c$ and $\nu_c \varphi_c$ are probability measures. Let $Q$ be a stochastic  transition function on ${\cal L}\times{\cal  L}$ defined as follows:
\begin{equation} \label{eq:Q}
Q(z,z')\ =\ \frac{1}{\lambda_c\varphi_c(z) } P_{\cal L}(z,z') \varphi_c(z').
\end{equation}

Observe that $Q$ inherits irreducibility and being aperiodic  from $P_{\cal L}$. As a result, $Q$ is ergodic, and we denote its unique stationary distribution by $\pi^Q$. Recall that from the definition of a stationary distribution,  $\pi^Q Q = \pi^Q$, if $\pi^Q$ is considered as a row vector, and it immediately follows that
\begin{equation}
\label{eq:piQ}
\pi^Q(z) \ = \ \nu_c (z) \varphi_c(z).
\end{equation}
We also define the marginal of the first coordinate $\pi^Q_1$ by letting
\begin{equation} \pi^Q_1 (x) \ = \ \sum_{b,b'}  \pi^Q (x,b,b').\end{equation}

Next we fix  some notation. We write  $P_\mu$ for the distribution of the Markov chain $Z$ under $P$ with initial distribution $\mu$, and $E^P_\mu$ for the corresponding expectation. When $\mu$ is a point mass $\delta_z$, we  denote this with $z$ as a subscript  instead of the notationally correct but more cumbersome $\delta_z$. We also define the analogous expressions with $Q$ instead of $P$.


The following result identifies the uniform distribution $Q^n$ with the distribution of the Markov chain $Z$ under $Q$.
\begin{theorem}
\label{th:structure}
Let $f=f (Z_0,\dots,Z_n)$ be a complex-valued random variable. Then
 \begin{equation} E^{Q_n}( f )\ = \ \frac{ E_{\tilde \varphi_c}^Q \left ( \frac{f}{\varphi_c(Z_n)}\right) }{E_{\tilde \varphi_c}^Q\left ( \frac{1}{\varphi_c(Z_n)}\right)},\end{equation}
  where $\tilde\varphi_c$ is the probability measure given by $\varphi_c$ conditioned on ${\cal L}_0$ in \eqref{eq:L_0}.
\end{theorem}

The theorem has a nice and simple interpretation in terms of the Markov chain corresponding to $Q$ pinned at time $n$. Specifically, if $D$ is a random variable on the same probability space as $Z$, independent of $Z$ and satisfying $Q(D=z) = \frac{c}{\varphi_c(z)}$, where $c$ is a normalizing constant to make the righthand side a probability mass  function, then we can restate the theorem as
\be  E^{Q_n} (f)\ =\ \frac{ E^{Q}_{\tilde \varphi_c} (f \ch_{\{ Z_n=D\}})}{Q_{\tilde \varphi_c} (Z_n =D)}.\ee

In other words, $Q^n$ is simply the distribution of $Q$ starting from $\tilde \varphi_c$, pinned at time $n$ to the randomly selected point $D$. Note that the dependence on $n$ is only through the time of the pinning, and this means that in order to study the sequence of probability measures $(Q_n)$, one only needs to study $Z$.

\begin{proof}[Proof of Theorem \ref{th:structure}]
Observe that if $z_0\in {\cal L}_0$ and $z_1,\dots,z_n  \in {\cal L}$, then
\begin{align}
P_{z_0}( \prod_{j=0}^{n} \{ Z_j = z_j\} , \tau> n) & \ = \  \prod_{j=0}^{n-1} P(z_j,z_{j+1})\nonumber\\
&\ = \ \lambda_c^n  \prod_{j=0}^{n-1} \varphi_c(z_j)Q(z_j,z_{j+1}) \frac{1}{\varphi_c(z_{j+1})} \nonumber\\
& \ = \ \lambda_c^n \varphi_c (z_0) Q _{z_0}( \prod_{j=0}^{n} \{ Z_j =z_j\} ) \frac{1}{\varphi_c (z_{n})},
\end{align}
and otherwise $P_{z_0}( \prod_{j=0}^{n} \{ Z_j = z_j\} , \tau> n)=0$. In particular, if $f=f(Z_0,\dots,Z_n)$ is a complex valued random variable, then
\begin{align} E^P ( f , \tau > n ) &\ = \ \sum_{z_0 \in {\cal  L}_0} E^P(f, \tau>n, Z_0=z_0) \ = \ \sum_{z_0  \in {\cal L}_0} E^P (\ch_{\{Z_0=z_0\} }  f(z_0,\dots,Z_n),\tau >n)\nonumber\\
& \ = \ \sum_{z_0\in {\cal L}_0} P(Z_0=z_0) E_{z_0}^P( f (Z_0,\dots,Z_n), \tau>n) \nonumber\\
& \ = \ \lambda_c^n \sum_{z_0\in {\cal L}_0} P(Z_0=z_0) \varphi_c (z_0)  E^Q_{z_0}  \left ( \frac{f }{\varphi_c(Z_n)}\right).
\end{align}
Since $P$ is uniform, it follows that $P(Z_0=z_0)$ is constant on ${\cal L}_0$, and the result follows. \end{proof}

Next we consider limits. The following provides sufficient conditions under which $Q_n$ expectations and expectations with respect to $Q$ are asymptotically equivalent.

\begin{prop}
\label{pr:asymp}
Suppose that for   $n\in\Z_+$,  $f_n (Z_0,\dots,Z_n)$ is a complex-valued random variable, and   $(j_n:n\in\Z_+)$ is a subsequence of $\Z_+$ such that
\begin{enumerate}
\item $\min (j_n, n -j_n) \to \infty$,
\item  $E_{\tilde \varphi_c} ^Q |f_n - f_{j_n}| \to 0$.
\end{enumerate}
Then
\begin{equation} | E^{Q_n} f_n - E^Q_{\tilde\varphi_c}  f_n | \ = \ o(1) \max (|E^Q_{\tilde\varphi_c} (f_n)|,1) .\end{equation}
\end{prop}

\begin{proof}
Because of condition (2), we have
 \begin{equation} E^{Q_n} \left ( f_n\right)  \ = \ \frac {E_{\tilde  \varphi_c}^Q \left (  \frac{f_{j_n} }{\varphi_c (Z_n)}\right)}{E_{\tilde \varphi_c}^Q\left (  \frac{1}{\varphi_c(Z_n)}\right)} + o(1).\end{equation}
 Then, by the Markov property,
 \begin{equation} E_{\tilde  \varphi_c}^Q \left (  \frac{f_{j_n} }{\varphi_c (Z_n)}\right) \ = \ E_{\tilde\varphi_c}^Q \left (  f_{j_n} E_{Z_{j_n}} \left ( \frac{1}{\varphi_c (Z_{n-j_n})}\right)\right).\end{equation}
  The ergodicity of $Z$ under $Q$ and the fact that $n-j_n\to\infty$ guarantee that $ E_{Z_{j_n}}^Q\left (  \frac{1}{\varphi_c (Z_{n-j_n})}\right)=E_{\pi^Q} \frac{1}{\varphi_c}  +o(1)=\|\nu_c\|_1+o(1)$. Thus
  \begin{eqnarray} E^{Q_n} (f_n) & \ = \ &  \frac{  (\|\nu_c\|_1 + o(1)) E_{\tilde \varphi_c}^Q ( f_{j_n})}{\|\nu_c\|_1 + o(1)}+o(1) \nonumber\\ & \ = \ & (1 + o(1)) E_{\tilde \varphi_c}^Q (f_{j_n} ) + o(1)\ = \ (1 + o(1))E_{\tilde \varphi_c}^Q (f_n) + o(1).\end{eqnarray}
\end{proof}

For applications, it would be useful to know more about $Q$.  It turns out that the underlying structure is determined by the matrix $C$, which we now describe. Let $C$ be the $L\times L$ matrix given by $C=(C_{i,j})$, $C_{i,1} = c_i$ and $C_{i,i+1} = 1$, and all other entries equal to $0$:
\be
\label{eq:C}
C\ =\
\left ( \begin{array}{ccccc}
 c_1 & 1 & 0 & \cdots \\
 c_2 & 0 & 1 & 0 & \cdots  \\
 \vdots & 0 & \cdots \\
 c_{L-1} & 0 & \dots & & 1\\
c_L & 0 & \dots  & & 0
\end{array} \right).
\ee
 Let $\lambda_C$ denote the Perron eigenvalue of $C$, $\varphi_C$ a corresponding positive right eigenvector and $\nu_C$ a corresponding left eigenvector.  A straightforward computation gives the following.

\begin{lemma}
\label{lem:char}
 Let $C$ be as in \eqref{eq:C}.  Then
 \begin{enumerate}
 \item the characteristic polynomial of $C$ is $  \lambda^L - \sum_{j=1}^L c_j \lambda^{L-j}$;
  \item  up to multiplicative constants: $\nu_C(b) = \lambda_C^{-b}$ and  $\varphi_C(b') = \lambda_C ^{b'} - \sum_{j=1}^{b'-1}c_j \lambda_C^{b'-j} $.
 \end{enumerate}
\end{lemma}

With this lemma we obtain a description of $Q$.

\begin{prop} Let $C$ be as in \eqref{eq:C}, and let $\lambda_c,\nu_c,\varphi_c$, respectively, be the Perron eigenvalue, and corresponding left and right eigenvectors for $P_{\cal L}$, the restriction of the transition function $P$ to ${\cal L}$, normalized so that $\varphi_c$ and $\nu_c\varphi_c$ are  probability distributions. Then:
\label{pr:quantities}
 \begin{enumerate}
 \item $\lambda_c =\frac{  \lambda_C}{(\max c_i +1)L}$.
 \item  There exist positive constants $K_1,K_2$ such that $\varphi_c (a,b,b') =K_1  \varphi_C (b')$ and $\nu_c(a,b,b') = K_2\nu_C(b)$. In particular, $ \pi^Q (a,b,b') = K_1 K_2 \nu_C(b) \varphi_C(b')$, and $K_1 K_2 = \frac{1}{\lambda_C \sum_{b=1}^L \nu_C(b) \varphi_C(b)}$.

\item  $ Q((a,b,b'),(a',b',b'') ) =  \frac{\varphi_C (b'')}{ \lambda_C\varphi_C (b')}$  for allowed transitions and is $0$ otherwise. \\
Furthermore, allowed transitions satisfy either of the following:
\begin{enumerate}
\item $b''= 1$ and then the probability of the transition is  $\frac{\varphi_C(1)}{\lambda_C\varphi_C(b')}$;
\item $b''=b'+1$ and then the probability of the transition is  $1- \frac{  \varphi_C(1) c_{b'}}{\lambda_C \varphi_C (b')}$.
\end{enumerate}
 \end{enumerate}
\end{prop}

\begin{example}
\label{ex:zeck_quant}
For the standard Zeckendorf decomposition, we have:
\begin{enumerate}
\item $ C = \left ( \begin{array}{cc}  1 & 1 \\ 1 & 0 \end{array} \right)$. In particular,
\begin{enumerate}
\item The characteristic polynomial is $\lambda^2-\lambda -1$, and $\lambda_C=  \phi$, where $\phi$ is the golden ratio $\phi = \frac{1+\sqrt{5}}{2}$.
\item  $\nu_C (b) = \phi^{-b}$, and $\varphi_C(b') = \phi^{2-b'}$.
\end{enumerate}
\item ${\cal L} = \{(0,1,1),(0,2,1), (1,1,2)\}$.  Identifying these states as $1$, $2$ and $3$ in the order written, then
\begin{enumerate}
\item
$  Q = \left ( \begin{array}{ccc}  \frac{1}{\phi} & 0 & 1- \frac{1}{\phi}  \\
\frac{1}{\phi} & 0 & 1- \frac{1}{\phi}  \\  0  & 1 & 0  \end{array} \right)$,
\item  $\pi^Q(0,1,1) = \frac{\phi}{2+\phi},~\pi^Q(0,2,1) = \frac{1}{2+\phi} \pi^Q(1,1,2) = \frac{1}{2+\phi}$, and \\
$\pi^Q_1(0) = \frac{1+\phi}{2+\phi},~\pi^Q_1(1) = \frac{1}{2+\phi}$,
\item  $\varphi_c  = \frac{1}{2\phi+1} \left (\phi,\phi,1 \right)^t$, and 
\item $\nu_c = \frac{1}{\phi+2}\left (2\phi+1,\phi+1,2\phi+1\right)^t$. 
\end{enumerate}

\end{enumerate}
\end{example}

\begin{proof}[Proof of Proposition \ref{pr:quantities}]~\\

1.  The first part is a straightforward calculation. \\

2. Observe that for the row of $P$ corresponding to transition from $(a,b,b')$, we have exactly $|{\cal S}_1|\times |{\cal S}_2|=(\max_i c_i +1)L$ allowed sites to transition to, and due to the choice of uniform distribution,  all are of equal probability. As  $P$ is stochastic,  its nonzero entries are equal to $\gamma = \frac{1}{(\max c_i +1)L }$.  We first study the restriction $P_{\cal L}$  of  $P$  to ${\cal L}\times {\cal L}$. Recall that the elements of $\cal L$ are of the form $(x,k,1)$, where $x < c_k$ or $(c_k,k,k+1)$ where $k=1,\dots,L-1$. For each $(a,b,b')\in {\cal L}$, $P_{\cal L}$ has a corresponding row, listing all transitions from $(a,b,b')$. We will count the number of such non-zero entries according to the value of $b'$. If $b'\in \{1,\dots, L-1\}$ then there are $1+c_{b'}$ transitions: one to the site $(c_{b'},{b'},b'+1)$ and $c_{b'}$ to $(x,b',1)$ where $x \in \{0,\dots, c_{b'}-1\}$. If $b' = L$ then there are only $c_L$ allowed transitions, all of which are of the second kind.


We define a  function $\varphi$ on ${\cal L}$  by letting $\varphi (a,b,b') = \varphi_C (b')$. Fix $(a,b,b')\in A$. If $b' < L$, then according to the allowed transitions listed above, we have \begin{equation} P_{\cal L} \varphi (a,b,b') \ = \ \gamma ( \varphi_C (b'+1) + c_{b'} \varphi_C (1)) \ = \ \gamma (C \varphi_C )(b') \ = \ \gamma \lambda_C \varphi(a,b,b').\end{equation} Similarly, if $b'=L$, then $P_{\cal L} \varphi(a,b,L) = \gamma c_L \varphi_C (1) = \gamma \lambda_C \varphi(a,b,L)$. Thus $\gamma \lambda_C=\lambda_c$, the Perron root for  $P_{\cal L}$, and $\varphi$ is a corresponding positive eigenvector. Next we want to find the corresponding  left-eigenvector for  $P_{\cal L}$. To do that, let $D$ be the transpose of $C$, and let $\nu_C$ be a Perron eigenvector. Define $\nu_c (a,b,b') := \nu_C (b)$. If  $b \in \{2,\dots, L\}$,  then there is exactly one allowed transition to it, that is from $(c_{b-1},b-1,b)$. As a result, $\nu_c P_{\cal L} (a,b,b') = \gamma \nu_c (c_{b-1},b-1,b) = \gamma  (D \nu_C)(b) =\gamma \lambda_C  \nu_c (a,b,b')$.
Next, if $b=1$, then the allowed transitions are from $(x,k,1)$ where $k=1,\dots, L$ and $x \in \{0,\dots,c_k-1\}$. We obtain
$ \nu_c P_{\cal L} (a,1,b') = \gamma \sum_{k=1}^L c_k \nu_C (k)  = \gamma (D \nu_C)(1)  =\gamma \lambda_C \nu_c (a,1,b')$.\\

The formula for $\pi^Q$
follows directly from \eqref{eq:piQ} and the preceding identities, while the formula for $K_1K_2$ follows from the calculation below.
\begin{align} \sum_{a,b,b'} \pi^Q(a,b,b') & \ = \  \sum_{a,b} \pi^Q (a,b,1) + \sum_{a,b} \pi^Q( a,b,b+1) \nonumber\\
 & \ = \ K_1K_2 \left(  \sum_{b=1}^{L}  c_b \nu_C (b) \varphi_C(1)+ \sum_{b=1}^{L-1} \nu_C(b) \varphi_C(b+1)\right)\nonumber\\
 & \ = \ K_1K_2 \sum_{b=1}^{L}  \nu_C (b) \left (  c_b \varphi_C(1) + \varphi_C(b+1)\right) \nonumber\\
& \ = \ K_1 K_2  \lambda_C \sum_{b=1}^L \nu_C(b) \varphi_C (b).
 \end{align}

\noindent 3. This follows from \eqref{eq:Q} and parts 1. and 2. \end{proof}

\section{Additive functionals}
\label{sec:additive}

\subsection{General Theory}\label{sec:additive_theory}
In this section we will study some theoretical aspects of large-time behavior of additive functionals of an ergodic finite-state Markov chain, under a change of measure which generalizes the way $Q_n$ was obtained from $Q$. The assumptions in this section are the following:

\begin{defi}
\label{def:additive}
Let $Z=(Z_n:n\in\Z_+)$ be an irreducible and aperiodic  Markov chain on the finite state space ${\cal L}$ with transition function $Q$.
Let $\varphi:{\cal L} \to (0,\infty)$ be a positive function, and let $\mu$ be a probability distribution on ${\cal L}$. For every $n\in\Z_+$, let $Q_n$ be a probability measure on $\sigma (Z_0,\dots,Z_n)$ given by
\be Q_n (A)\ =\ \frac{E^Q_{\mu} \left (   \frac{\ch_A}{\varphi(Z_n)}\right)}{E^Q_{\mu}\left (  \frac{1}{\varphi(Z_n)}\right)},~A \in \sigma (Z_0,\dots,Z_n).\ee
\end{defi}
We will consider the behavior of  additive functionals of the form $S_n = \sum_{j=0}^n g (Z_j)$ where $g :{\cal L} \to \C$ under $Q_n$ as $n\to\infty$. In the context of generalized Zeckendorf decompositions, an example for an additive functional is the number of, say, nonzero digits in the decomposition. In the next section, we show that gaps in the decomposition can be viewed as additive functionals of some Markov chain, so we can treat them with the same tools. \\

We need to fix some notation. Functions on ${\cal L}$ will interchangeably be viewed as column vectors. As an example, if  $g$ is such a function then $Qg$ is to be identified as the function or, equivalently the column $f$ vector given by $f(z) = \sum_{z' \in {\cal L}} Q(z,z') g(z')$. We will  write $hg$ for the product of such two functions, namely $hg$ is the function given by $(hg)(z) = h(z) g(z),~z\in {\cal L} $. In addition, $h (Qg)$ means the product of the function $h$ and the function $Qg$, not their scalar product. \\

Let $\pi^Q$ denote the stationary distribution for $Q$. Recall that $I-Q$ is invertible on the $Q$-invariant subspace of  $V$, where $V =\{g:E_{\pi^Q}  g(z)=0\}$. We denote this inverse by $Q^{\#}$, and extend it to all functions by letting $Q^{\#} {\ch} =0$. This is the only choice that guarantees that $Q$ and $Q^\#$ commute, and $Q^\#$ is known as the group inverse of $Q$.  It is well-known that
\begin{equation}
\label{eq:Qsharp}
 \sum_{j=0}^\infty E_z^{Q} (g(Z_j) -E_{\pi^Q} g)\ = \ (Q^\# g) (z).
\end{equation}

Our first result is the following.

  \begin{theorem}  \label{th:additive_functionals}
Let $g:{\cal L} \to \C$. Let $\tilde g= g - E_{\pi^Q} g$,  and  $\tilde S_n = \sum_{j=0}^{n} \tilde g(Z_j) $. Then
  \begin{eqnarray} \label{eq:lekkerkerker2} E^{Q_n}\tilde S_n &\ = \ & E_{\mu} (Q^\#  g) + \frac{ E_{\pi^Q} \tilde g  (Q^\# \frac{1}{\varphi})}{E_{\pi^Q} \frac{1}{\varphi} } + o(1) \end{eqnarray}
  \begin{eqnarray}
 \label{eq:sec_mom} E^Q_{\pi^Q} \tilde S_n^2 &\ = \ & (n+1)E_{\pi^Q} \left (\tilde g ((2Q^\# - I) \tilde g)\right)+o(1)\mbox{ and } E^{Q_n} \tilde S_n^2\ =\ (1+o(1))E^Q_{\pi^Q} \tilde S_n^2.\ \ 
 \end{eqnarray}
  \end{theorem}

  \begin{proof}
We will first prove \eqref{eq:lekkerkerker2}. From Theorem \ref{th:structure}  with $f= \tilde S_n$, and the Markov property, we have that
\begin{align} E^Q_{\mu} \frac{1}{\varphi(Z_n)}\times   E^{Q_n} \tilde S_n & \ = \ \sum_{j=0}^n E^Q_{\mu} \tilde g (Z_j) E_{Z_j} \frac{1}{\varphi(Z_{n-j})}\nonumber\\
 & \ = \ \usb{(I)}{\sum_{j=0}^n E^Q_{\mu} \tilde g (Z_j)\left (  E_{Z_j}^Q \frac{1}{\varphi(Z_{n-j})} - E_{\pi^Q} \frac{1}{\varphi}\right)}\nonumber\\
   &\ \ \ \ \ \ \ +\ E_{\pi^Q} \frac{1}{\varphi}\usb{(II)}{ \sum_{j=0}^n  E^Q_{\mu} \tilde g (Z_j) }.
   \end{align}
   By \eqref{eq:Qsharp}, $(II)\to E_{\mu} (Q^\# \tilde g)=E_{\mu} Q^\# g$, because $Q^\#$ maps constant function to $0$.  In order to estimate $(I)$, we recall that from the exponential ergodicity of irreducible finite state Markov chains, there exists $\rho \in(0,1)$ and $c_1>0$, such that for every function $h$ and $k\in \Z_+$,
   \begin{equation}
   \label{eq:exp_ergo}  \sup_{z} | E_z^Q h (Z_k) - E_{\pi^Q} h|\ \le\ c_1 \|h\|_\infty \rho^k.
   \end{equation}

Letting $h(z) = \frac{1}{\varphi(z)}-E_{\pi^Q} \frac{1}{\varphi}$, we have that $E_{\pi^Q} h =0$. This allows us to rewrite $(I)$ as $ \sum_{j=0}^n E_{\mu} \tilde  g (Z_j) E_{Z_j} h (Z_{n-j})$.
    In order to estimate this sum, we break it into two parts. First
    \begin{equation}\left|\sum_{j=0}^{\lfloor n/2 \rfloor } E_{\mu}^Q \tilde g (Z_j)  E_{Z_j} h (Z_{n-j}) \right| \ \le \  \|\tilde g\|_\infty \sum_{j=0}^{\lfloor n/2 \rfloor } \sup_{z}  |E_{z} h (Z_{n-j}) | \ \le \  c_1\|\tilde g\|_\infty \|h\|_\infty \rho^{n/2} n/2\to 0,\end{equation}
    where the last inequality follows from \eqref{eq:exp_ergo}.
    Next, let $h_{k}(z) = \tilde g(z) E_z h (Z_{k})$. Then
    \begin{equation} \sum_{j={\lfloor n/2 \rfloor +1}}^n E_{\mu}^Q \tilde g(Z_j) E_{Z_j} \frac{1}{\varphi(Z_{n-j})} \ = \  \sum_{j={\lfloor n/2 \rfloor +1}}^n  E_{\mu}^Q  h_{n-j} (Z_j) .\end{equation}
    Applying \eqref{eq:exp_ergo} to each of the functions  $h_{k}$, and observing that $\|h_k\|_\infty \le \|\tilde g\|_\infty \|h\|_\infty$, it follows that   for $j\ge \lfloor n/2 \rfloor +1$,
    \begin{equation}|E_{\mu}^Q  h_k (Z_j) - E_{\pi^Q} h_k  |\ \le \  c_1 \|\tilde g\|_\infty\|h\|_\infty\rho^{n/2}.\end{equation}
    Also, since $\pi^Q$ is the stationary distribution for $Q$, we have that $ E_{\pi^Q} h_k = E_{\pi^Q}^Q h_k(Z_j)$, and as a result
    \begin{equation}\sum_{j={\lfloor n/2 \rfloor +1}}^n \left ( E_{\mu}^Q  h_{n-j} (Z_j) - E_{\pi^Q}^Q   h_{n-j} (Z_j) \right)\ \le \  c_1 \|\tilde g\|_\infty\|h\|_\infty  \rho^{n/2}n/2 \to 0.\end{equation}
    In addition, $E_{\pi^Q}^Q h_{n-j} (Z_j) = E_{\pi^Q} \tilde g(Z_0) E_{Z_0} h (Z_{n-j})$, and therefore
     \begin{equation}\sum_{j=\lfloor n/2 \rfloor +1}^n E_{\pi^Q}^Q   h_{n-j} (Z_j) \ = \ \sum_{k=0}^{n- \lfloor n/2 \rfloor -1} E_{\pi^Q}^Q \tilde g (Z_0) E_{Z_0} h (Z_k).\end{equation}
     Since by our choice  $E_{\pi^Q} h = 0$, it follows from \eqref{eq:Qsharp} that the righthand side is equal to $E_{\pi^Q}  \tilde g (Q^\# h)+o(1)$.
    As a result, $ (I) = E_{\pi^Q} \tilde g (Q^\#  \frac{1}{\varphi})+ o(1)$, completing the proof of \eqref{eq:lekkerkerker2}. \\

    We turn to proving \eqref{eq:sec_mom}.  We first prove the first equality.
  \begin{align}
  E_{\pi^Q}^Q \left ( \tilde S_n^2  \right)&\ = \ \sum_{j=0}^n E_{\pi^Q}^Q \tilde g ^2 (Z_j) + 2 \sum_{0\le j < k\le n} E_{\pi^Q} \tilde g(X_j) \tilde g (X_k)\nonumber\\
  & \ = \ (n+1)E_{\pi^Q}^Q \tilde g ^2 + 2 \sum_{0\le j < k\le n} E_{\pi^Q}^Q \tilde g(X_0) E_{X_0}^Q \tilde g (X_{k-j})\nonumber\\
  & \ = \ -(n+1)E_{\pi^Q} \tilde g^2 + 2 \sum_{j=0}^n \sum_{k=j}^n  E_{\pi^Q}^Q \tilde g(X_0) E_{X_0}^Q \tilde g (X_{k-j}) \nonumber\\
  & \ = \ - (n+1) E_{\pi^Q} \tilde g ^2 + 2 \sum_{j=0}^n E_{\pi^Q}^Q \tilde g(X_0)\left (  \sum_{k=0}^{n-j} E_{X_0}^Q \tilde g (X_k)\right) \nonumber\\
  & \ = \ -(n+1) E_{\pi^Q} \tilde g^2 + 2 \sum_{j=0}^n E_{\pi^Q} \tilde g Q^\# \tilde g - 2 \usb{(*)}{\sum_{j=0}^n E_{\pi^Q}^Q \left ( \tilde g (X_0) \sum_{k > n-j} E^Q_{X_0} \tilde g (X_k)\right)}\nonumber\\
  & \ = \ (n+1) E_{\pi^Q} \tilde g (2Q^\#-I) \tilde g + (*).
  \end{align}
  Observe that by exponential ergodicity, \eqref{eq:exp_ergo}, $|E_z^Q \tilde g (X_k) |\le c_1 \|\tilde g\|_\infty  \rho^k$, uniformly over $z$, and so
  \begin{equation}|(*)| \ \le\ c_1 \|\tilde g\|_\infty^2 \sum_{j=0}^n \frac{\rho^{n-j+1}}{1-\rho} \ \le\ c_1 \|\tilde g\|_\infty^2  \frac{1}{(1-\rho)^2}\ =\ O(1).\end{equation}
  This completes the proof of the first equality in \eqref{eq:sec_mom}. It remains to the asymptotic equivalence of $E^Q_{Q_n} \tilde S_n^2$ and $E^Q_{\mu} \tilde S_n^2$. This, again, follows from the exponential ergodicity, as we now explain. We have
  \be \tilde S_n^2\ = \ \tilde S_m^2 + 2 \tilde S_m (\tilde S_n -\tilde S_m) + (\tilde S_n -\tilde S_m)^2.\ee
  From the Markov property and exponential ergodicity \eqref{eq:exp_ergo}, it follows that
  \be |E_{\mu} (S_n -S_m)^2 -E_{\pi^Q} \tilde S_{n-m}^2 |\ \le\ c_1 \|\tilde g\|_\infty n^2 \rho^{m}.\ee
Choose $m = c \ln n$ for $c=4/\ln (1/\rho) $. It follows  that righthand side tends to $0$ as $n\to\infty$. In particular,    $E_{\mu} (S_n -S_m)^2 \le c_2 n$.
Next, observe that
  $ E_{\mu} \tilde S_m^2 \le \|\tilde g\|_\infty^2 m^2$, and by Cauchy-Schwarz,
  $|E_{\mu} \tilde S_m (\tilde S_n -\tilde S_m)| \le\sqrt{  E_{\mu} \tilde S_m^2 }\sqrt{E_{\mu} (\tilde S_n -\tilde S_m)^2}\le c_3 m \sqrt{n}$.
  In summary, for all $n$ large enough,
  \be |E_{\mu}\left (  \tilde S_m^2 + 2 \tilde S_m (\tilde S_n -\tilde S_m)\right) |\ \le\ c_4 (\ln n)^2   \sqrt{n}\ \le\ c_4 n^{3/4}.\ee
   In particular,
   \be |E_{\mu} \tilde S_n^2 - E_{\pi^Q}^Q \tilde S_{n-m}^2 |\ \le\ c_4 n^{3/4},\ee
   so that
   \be E_{\mu} S_n^2 \ =\ (1 + o(1)) n E_{\pi^Q}  \tilde g (2Q^\#-I) \tilde g,\ee
    and the claim is proved.  \end{proof}

 We turn to laws of large numbers and central limit theorems for additive functionals.

 \begin{theorem}~
 Under the same assumptions of Theorem \ref{th:additive_functionals} we have:
  \label{th:LLN_CLT}
\begin{enumerate}
\item  Weak Law of Large Numbers:  For $\epsilon>0$, $\lim_{n\to\infty} Q_n \left (|\frac{ \tilde S_n }{n+1} |> \epsilon \right)=0$. \item Central Limit Theorem:
  $ Q_n \left( \frac{\tilde S_n}{\sqrt{n+1}} \le x~\right) \Rightarrow  P( Y \le x)$ where $Y\sim N(0,\sigma^2)$,
  and  $\sigma^2 =  E_{\pi^Q} \tilde g ((2Q^\# -I)\tilde g)$.
    \end{enumerate}
 \end{theorem}

\begin{proof}[Proof of Theorem \ref{th:LLN_CLT}]
The Weak Law of Large Numbers follows from Chebychev's inequality and the asymptotic estimate for $E^{Q_n} \tilde S_n^2$ given in Theorem \ref{th:additive_functionals}:
\be Q_n\left  (|\frac{ \tilde S_n}{n+1} |\ > \ \epsilon\right)\ \le\ \frac{ E^{Q_n} \tilde S_n^2 }{(n+1)^2 \epsilon^2} \ = \ \frac{E_{\pi^Q} \tilde g (2Q^\# -I) \tilde g }{(n+1)\epsilon^2}\to 0,\mbox{ as } n\to\infty.\ee

We now prove the  Central Limit Theorem. To do this we apply Proposition \ref{pr:asymp} with $j_n = n- \lfloor \ln n\rfloor $ and
   \begin{equation} f_n \ = \ \exp\left (  \frac{i \theta}{\sqrt{n+1}}  \tilde S_n \right).\end{equation}
   Observe that the choice of $j_n$ guarantees that condition 1. in the proposition holds. Next,
\begin{eqnarray} E^Q_z |f_n -  f_{j_n}  | & \ \le \ & E_z^Q  | 1- E_{Z_{j_n}}^Q e^{\frac{ i \theta}{\sqrt{n+1}} \tilde S_{n-j_n} }| \nonumber\\ & \ \le \ & \max_{z} \left ( |1- E_z^Q\cos \left( \frac{\theta \tilde S_{n-j_n}}{\sqrt{n+1}}\right) | + |E_z^Q \sin \left(\frac{\theta \tilde S_{n-j_n}} {\sqrt{n+1}} \right)|\right).\end{eqnarray}
Since $|S_{n-j_n}| =O( \ln n)$, it follows from bounded convergence that $ \sup_{z} E^Q_z |f_n -  f_{j_n} | \to 0$, and so condition 2.  holds.
Finally, we recall from the Central Limit Theorem for additive functionals of finite state Markov chains (e.g. \cite{maxwell_woodroofe},\cite[Theorem 5]{benari_neumann}, that \begin{equation} E_{\mu}^Q (f_n) \to e^{-\frac{ \sigma^2}{2}},\end{equation} where $\sigma^2 = \lim_{n\to\infty} \frac{1}{n+1} E_{\pi^Q}^Q \left ( \tilde S_n ^2  \right)$. The result now follows from Theorem \ref{th:additive_functionals}.    \end{proof}


\subsection{Application to Zeckendorf Decompositions}
\label{sec:application_zeckendorf}

In this section we show how the results obtained in Section \ref{sec:additive_theory} apply to generalized Zeckendorf decompositions. In particular we will show that the generalized Lekkerkerker's theorem (Theorem \ref{th:gen_lek}) and the corresponding Central Limit Theorem are specials cases to Theorem \ref{th:additive_functionals}-1 and Theorem \ref{th:LLN_CLT}-2. We will also carry out explicit computations for the standard Zeckendorf decomposition, where all quantities are easily computable. \\

In order to apply the results in  the context of generalized Zeckendorf decomposition, in Definition \ref{def:additive} we identify  ${\cal L}$, $Z$ and $Q$ in the definition as the same quantities defined in  Section \ref{sec:prob}, and also set $\varphi=\varphi_c$, and $\mu= \tilde\varphi_c$, where $\varphi_c$ and $\tilde \varphi_c$ are as in  Section \ref{sec:prob}. With these choices, the measure $Q_n$ of Definition \ref{def:additive} coincides with $Q_n$ of Section \ref{sec:prob}. \\

Recall $k(N)$, the number of nonzero summands in the generalized Zeckendorf decomposition of $N$, defined in \eqref{eq:kN_def}. Let $g:{\cal L} \to \{0,1\}$ be defined as $g(x,j,j')=1$ if and only if $x>0$. Then if $N \in [G_{n+1},G_{n+2})$, from \eqref{eq:decomp2} we have that  that $k(N) =S_n$, where $S_n$ is the additive functional $S_n =  \sum_{j=0}^n g (Z_j)$.   Observe that $E_{\pi^Q} g = 1- \pi_1(0)$, and so $\tilde g= g - 1 + \pi_1(0)$. Furthermore, since $\pi^Q(z) = \varphi_c (z) \nu_c(z)$, it follows that $E_{\pi^Q} \frac{1}{\varphi_c} = \|\varphi\|_1$. The following therefore follow immediately from Theorem \ref{th:additive_functionals} and Theorem \ref{th:LLN_CLT}.

   \begin{cor} For generalized Zeckendorf decomposition:
  \label{cor:lekkerkerker}
  \begin{enumerate}
  \item Generalized Lekkerkerker's Theorem (Theorem \ref{th:gen_lek}):
      \be E^{Q_n} k(N)\ =\ C_{\rm Lek} (n+1)  + d\ee  where
      \be C_{\rm Lek} \ = \ 1-\pi_1(0),~d \ = \ E_{\tilde \varphi_c} Q^\#(1-\delta)+\frac{ E_{\pi^Q} (1-\delta)  (Q^\# \frac{1}{\varphi_c})}{\|\nu_c\|_1},  \ee
      \item Variance:
      \be E^{Q_n} (k(N) - C_{\rm Lek} (n+1) )^2\ =\ (1+o(1)) (n+1) \sigma^2\ee
      where
   \be \sigma^2\ =\ E_{\pi^Q} \tilde g((2Q^\#- I )  \tilde g).\ee
   \end{enumerate}
      \end{cor}

   \begin{cor} For generalized Zeckendorf decompositions we have
     \begin{enumerate}
   \item    Law of Large Numbers:
      \be  Q_n\left( \left | k(N) - C_{\rm Lek}(n+1) \right | > n \epsilon\right)\ \to\ 0;\ee

      \item Central Limit Theorem:
      \be Q_n\left( \frac{k(N) - C_{\rm Lek}(n+1) }{\sqrt{n+1}} \in \cdot\right)\ \to\ N(0,\sigma^2) \ee
      where $\sigma^2$ is as in Corollary \ref{cor:lekkerkerker}
      \end{enumerate}
     \end{cor}

In the remainder of the section we compute all constants above for the standard Zeckendorf decomposition. First we need to compute $Q^\#$.

\begin{example}
 For the standard Zeckendorf decomposition,
\begin{equation} Q^{\#} \ = \  \frac 15 \left ( \begin{array} {ccc} 5-\phi & \phi -4 & -1 \\ -\phi & \phi+1 & -1 \\
1-3\phi & 2\phi -2 & \phi +1 \end{array}\right).\end{equation}
\end{example}

To prove the identity, recall the expressions for $Q$ and $\pi^Q$ computed in Example \ref{ex:zeck_quant}. Let  $A=I-Q$, and let $v_1 = (0,1,-1)^t$, $v_2=(1,0,-\phi)^t$, and $v_3 = (1,1,1)^t$. Then $E_{\pi^Q} v_1 = E_{\pi^Q}  v_2=0$. Since $v_1$ and $v_2$ are linearly independent, it follows that they span the $A$-invariant space  $V =\{v: E_{\pi^Q} v=0\}$. In addition $A v_3=0$. Letting $q= 1- \frac{1}{\lambda_C}  = \frac{1}{\lambda_C^2}$, a straightforward  calculation shows that $A v_1 = q v_2 + (1+q) v_1$, and $A v_2 = v_2$. Thus  $v_1 = q v_2 + (1+q) Q^\# v_1$,  $Q^\# v_2 =v_2$ and $Q^\# v_3 = 0$. These determine $Q^\#$. \\

 Also, from Example \ref{ex:zeck_quant} we have that  $\pi_1(0) = \frac{\phi+1}{\phi+2}$,  $\tilde \varphi_c$ is a point mass, and $\|\nu_c\|_1 = \frac{5\phi+3}{\phi+2}$. In addition,  $\pi^Q = \frac{1}{\phi+2} \left(\phi ,1,1\right)^t$, and $\varphi_c = \frac{1}{2\phi+1} (\phi,\phi,1)^t$. Since  also $g= (0,0,1)^t$, we have   $Q^\# g  =\frac 15  \left (1-3\phi,2\phi-2,\phi+1 \right )^t$, and  $Q^\# \frac{1}{\varphi_c} = \frac{1}{5(\phi-1)}\left(-1, -1, \phi+1\right)^t$. As a result, we have the following.

   \begin{example}
   For the standard Zeckendorf decomposition:
   \begin{equation} C_{\rm Lek} \ = \ \frac{1}{\phi+2} \ = \ \frac{5 -\sqrt{5}}{10},\ \ \ \ ~d\ =\ \frac 35.\end{equation}
   \end{example}

We finally compute $\sigma^2$.   Clearly, $\tilde g=(0,0,1)^t - \frac{1}{2+\phi} (1,1,1)^t=\frac{1}{2+\phi} ( -1,-1,1+\phi)^t$.
It therefore follows that $\tilde g Q^\# \tilde g= \tilde g Q^\# (0,0,1)^t = \frac 15 \left ( \frac{1}{\phi+2} ,\frac{1}{2+\phi} ,(1-\frac{1}{2+\phi}) (\phi+1) \right)^t$, and so   the expectation is equal to
\begin{equation}
\label{eq:first} 2E_{\pi^Q}\tilde gQ^\# \tilde g \ = \ \frac 25  \left ( \frac{1+\phi+(1+\phi)^2}{(\phi+2)^2}\right)\ = \ \frac{2(\phi+2)}{25}.
\end{equation}
 Since $\tilde g^2  = (\frac{1}{(\phi+2)^2}, \frac{1}{(\phi+2)^2}, \frac{(\phi+1)^2}{(\phi+2)^2})^t= \frac {1}{5 (1+\phi)}( 1,1,(1+\phi)^2)^t$,
 it follows that
 \begin{equation}
 \label{eq:second}
 E_{\pi^Q} \tilde g^2 \ = \ \frac{1}{5(1+\phi)}  \frac{( \phi +1)+(\phi+1)^2}{\phi+2} \ = \ \frac 15.
 \end{equation}

We therefore have

\begin{example} For the standard Zeckendorf decomposition:
$\sigma^2 = \frac{2\phi-1}{25}  = \frac{\sqrt{5}}{25}$.
\end{example}


\section{Gaps in Zeckendorf Decomposition}
\label{sec:gaps}
\subsection{Gap Distribution}
\label{sec:gap_dist}
In this section we consider the asymptotic distribution of gaps between non-zero terms in the generalized  Zeckendorf decomposition. This will be an application of our results on additive functionals from the previous section. We will first prove a statement on an ``average" gap distribution, Theorem \ref{th:gap_dist}, and we will later prove convergence of empirical gap measures in probability, Theorem \ref{th:weaklim_prob}. Let us first  define the notion of a gap. We work under the same assumptions and notation as in Section \ref{sec:prob}. Suppose that $N\in \N$ admits a legal decomposition \eqref{eq:decomp2}  with $X_0>0$. Note that $X_j$  counts the repetitions of $G_{n-j+1}$, and if repeating more than $1$ times, we can view this as $X_j-1$ gaps of length zero. If  $X_j>0$, then we have a gap of length $1$  or larger, the length of the gap equal to  $\min\{k \ge 1: X_{j+k} >0\}$. Let $N_n (k)$ denote the number of gaps of length $k$ in the first $n$ digits, and let $N_n= \sum_k N_n(k)$. We define the gap distribution $\mu_n$ as a probability measure on $\Z_+$ given by

\begin{equation}\mu_n (k) \ = \ \frac{ E^{Q_n} N_n (k)}{E^{Q_n} N_n}.\end{equation}
To state the next theorem, let
\be \nu(k) \ =\ \lambda_C ^{-(k-1)} (1-\lambda_C^{-1}) \ee
denote the probability density of a geometric random variable with parameter $\lambda_C^{-1}$. We have

\begin{theorem} Let  $ H_1 = \{(0,b,1) \in{\cal L}\}$ and  $H_2 = \{(0,b+1,b+2)\in {\cal L}: c_b>0,c_{b+1} =0\}$.  For $z=(0,b+1,b+2) \in H_2$ we let
 \begin{align} r(b) &\ = \  \max\{j : c_{b+j} = 0\},~\nonumber\\
  \rho(b)   &\ = \  Q ( (0,b+r(b),b+r(b)+1),(0,b+r(b)+1,1))= \frac{\varphi_C (1)}{\lambda_C \varphi_C(b+r(b)+1)}\mbox{, and }\nonumber\\
 h (b,k) &\ = \  \begin{cases} 0 & k < r (b)+1 \\  1- \rho(b) & k = r(b)+1 \\ \rho (b)  \lambda_C^{-(k-r(b)-2)}(1-\lambda_C^{-1}) & k> r(b)+1.\end{cases}
 \end{align}
Then
\label{th:gap_dist}
\begin{enumerate}
\item $\lim_{n\to\infty} \frac{1}{n} E^{Q_n} N_n =M_{\pi^Q_1}$.
\item
\be \lim_{n\to\infty} \mu_n(k)\ =\ \begin{cases} 1- \frac{1-\pi^Q_1(0)}{M_{\pi^Q_1}}& k= 0 \\
\frac{ 1- \pi^Q_1(0) - \pi^Q(H_1) ( 1- \lambda_C^{-1}) -\sum_{z\in  H_2} \pi^Q (z) (1- \rho(z(2)))}{M_{\pi^Q_1}}& k=1.
\end{cases}\ee
\item For $k\ge 2$,
 \begin{align}
 \nonumber
 \label{eq:zzz}\lim_{n\to\infty} \mu_n (k) &\ = \  \frac{ \pi^Q (H_1) \nu(k-1)}{M_{\pi^Q_1}}.\nonumber\\
 &\ \ \ \ \ \ +\ \frac {\sum_{z \in H_2} \pi^Q ( z ) \left ( h (z(2)-1,k) -\rho(z(2))\nu (k-1)\right)}{M_{\pi^Q_1}}.
 \end{align}
 \end{enumerate}
 \end{theorem}
 Since $\sum_{k\ge 2}  \nu(k-1)=\sum_{k\ge 2} h(b,k) =1$, it follows that the limit $\lim_{n\to\infty} \mu_n(\cdot )$ is a probability measure, which we denote by $\mu_\infty$. A simple argument shows that a stronger result holds. For $n\in\N$, define the empirical gap distribution $\hat \mu_n$ as a random measure on $\Z_+$,   defined by
\be  \hat \mu_n (A)\ =\ \frac{ \sum_{k \in A} N_n(k) }{\max (N_n,1)}.\ee

We therefore have the following.

\begin{theorem}
\label{th:weaklim_prob}
For any $A\subset \Z_+$ and $\epsilon>0$,
\be \lim_{n\to\infty} Q_n (\left |\hat \mu_n (A) -  \mu_\infty (A)\right|>\epsilon)\ =\ 0.\ee
\end{theorem}
We comment that the expression for the  limit in Theorem \ref{th:gap_dist} is much simpler when $c_j>0$ for all $j=1,\dots,L$. In this case $H_2=\emptyset$. For the standard Zeckendorf, we have the following.

\begin{example}
For the standard Zeckendorf decomposition, $M_{\pi^Q_1} = \pi^Q_1 (1) = \frac{1}{\phi+2}$ and $\lambda_C = \phi$. Therefore
\begin{enumerate}
\item $\lim_{n\to\infty} \frac{1}{n} E^{Q_n} N_n =\frac{1}{\phi+2}$.
\item
$\lim_{n\to\infty} \mu_n(k)\ =\ \begin{cases} 0 & k= 0,1 \\
\phi^{-k} & k\ge 2.
\end{cases}$
\end{enumerate}
\end{example}

When some of the coefficients are zero, then some gaps of length $\ge 2$ are forced by the recurrence relation,  and taking this into account is the source of the lengthy expression in the theorem.

\begin{example} Consider the recurrence relation with $L=4$,  $c_1=1,c_2=c_3=0,c_4=2$. Then $\lambda_C $ is the largest  (real) root of  $\lambda^3 (\lambda-1) = 2$, $\lambda_C \thickapprox 1.5437$. We have \be  h(k)\ =\ \begin{cases} 0 & k <3 \\  \frac 12 & k=3 \\  \frac 12  \lambda_C^{-(k-4)} (1-\lambda_C^{-1}) & k \ge  4\end{cases}\ee
and
\be  \lim_{n\to\infty} \mu_n (k)\ =\ \begin{cases} 0 & k= 0 \\
  2-\frac{ \lambda_C^2+1}{3\lambda_C} & k=1
  \\  \frac{(\lambda_C-1)^2 - \lambda_C}{3\lambda_C}\nu (k-1) + \frac{ 2\lambda_C-1}{3\lambda_C} h(k) & k\ge 2.  \end{cases}
  \ee
  \end{example}

In this example,
\be {\cal L}\ =\ \left\{z^1 = (0,1,1),z^2=(1,1,2),z^3=(0,2,3),z^4=(0,3,4),z^5=(0,4,1),z^6=(1,4,1)\right\}.\ee There are no gaps of length $0$ as the coefficients immediately show. Gaps of length $1$ only appear in the form $(1,4,1)$ followed by $(1,1,2)$.  Larger gaps can be formed as follows.
\begin{itemize}
 \item Gaps of length $k\ge 2$  through a  sequence of the form $(1,4,1), (0,1,1),\dots,(1,1,2)$, with $(0,1,1)$ repeated $k-1$ times.
 \item Gaps of length $k\ge 3$ through a sequence beginning with  $(1,1,2),(0,2,3),(0,3,4)$, followed by $(1,4,1)$ if length is $3$, or  by $k-3$ repetitions of  $(0,1,1)$ followed by $(1,1,2)$ otherwise.
 \end{itemize}

The larger gaps of the second type are forced by the recurrence, in the sense that the condition $c_2=c_3=0$ implies  $Q( (1,1,2),(0,2,3)) = Q((0,2,3),(0,3,4))=1$, and so every time the sequence hits the state $ (1,1,2)$, a gap of minimal length $3$ occurs. Let us see how this is reflected in the formula. $H_1 = \{(0,1,1),(0,4,1) \}$ and  $H_2=\{(0,2,3)\}$. There's only one element in $H^2$ and therefore we omit the reference to $b$ in the functions $r,\rho,h$. So $r=2$, $\rho = Q( (0,3,4),(0,4,1)) =\frac 12$, and the expression for $h$ follows. \\

We now compute $\pi^Q$.  Let $p=\pi^Q(z^2)$. Since $Q(z^2,z^3)=Q(z^3,z^4)=1$, we have that $p=\pi^Q(z^3)=\pi^Q(z^4)$. Next, $Q(z^4,z^5)=Q(z^4,z^6)=\frac 12$, and so $\pi^Q (z^5)=\pi^Q (z^6) = p/2$. We also observe that \be \pi^Q (z^1)\ =\ \pi^Q (z^1) \lambda_C^{-1}+ \pi^Q (z^5) Q (z^5,z^1) + \pi^Q(z^6) Q (z^6,z^1)\ee  Therefore, $\pi^Q(z^1)= \frac{p}{\lambda_C-1}$.  Now we have $1 = \frac{p}{\lambda_C-1} + 4p $, so altogether, $p=\frac{\lambda_C-1}{4\lambda_C-3}$, and the expression for the limit of $\mu_n$ follow after some algebra.

\begin{proof}[Proof of Theorem \ref{th:gap_dist}]  For a real number $x$, let $x_+= \max (x,0)$. We begin with gaps of length $0$:
  \begin{equation}E^Q_z  N_n(0) \ = \ \sum_{j=0}^{n-1} (Z_j(1)-1)_+.\end{equation}
   The ergodicity of $Z$ under $Q$ implies that
  \begin{equation}
  \label{eq:N0} \lim_{n\to\infty} \frac{E^Q_z N_n(0)}{n} \ = \ \sum_{z=(x,j,j')} \pi^Q (z)(x-1)_+=M_{\pi^Q_1}-1+\pi^Q_1(0).
  \end{equation}
  Before moving to gaps of larger length, we consider the total number of jumps. We have
\begin{align}
\nonumber
 \frac 1n E_z^Q \sum_{k \ge 1} N_n (k) &\ = \  \nonumber
   \frac 1n E^Q_z \sum_{j=0}^{n-1} \ch_{\{Z_j(0)>0\}}\\
    \label{eq:N12}
   & \underset{n\to\infty} { \to} 1-\pi^Q_1(0),
 \end{align}
  and so from \eqref{eq:N0}, \eqref{eq:N12}
    \begin{equation}  \label{eq:N}
    \lim_{n\to\infty}  \frac 1n E^Q_z  N_n  \ = \ M_{\pi^Q_1}.
    \end{equation}
We move to calculation of gaps of length $\ge 2$. We will treat gaps of length $1$ last.   Let $k\ge 2$. Then
  \begin{align} \frac 1n E^Q_z N_n(k) &\ = \    \frac 1n E^Q_z \sum_{j=0}^{n-k}\ch_{\{Z_{j}(1)> 0\}}\left(\prod_{\ell=1}^{k-1} \ch_{\{Z_{j+\ell}(1)=0\}}\right) \ch_{\{Z_{j+k}(1) > 0\}}.
\end{align}
Let $B=\{(0,b,b')\in {\cal L}\}$. It therefore follows from the Markov property and ergodicity that
\be \lim_{n\to\infty} \frac1n E^Q_z N_n(k)\ =\ \sum_{z^0 \in A} \pi^Q (z^0)  f_B(z^0)\ee
where for $D\subset {\cal L}$ we have
\be f_D(z^0)\ =\  \left(  \sum_{z^1\in D, \dots,z^{k-1}\in B }  \prod_{\ell=1}^{k-1} Q(z^{\ell-1},z^\ell) \right) Q(z^{k-1},A). \ee

Letting
\begin{align}
B_0 &\ = \  \{(0,1,1)\}\nonumber\\
B_1 &\ = \ \{(0,b+1,1)\in {\cal L}: b\ge 1,~c_b >0\},\nonumber\\
B_2 &\ = \  \{ (0,b+1,b+2)\in {\cal L}: b\ge1,~c_b >0, c_{b+1} =0\}, \mbox{ and}\nonumber\\
B_3 &\ = \  \{(0,b+1,1)\in {\cal L}: b \ge 1,~c_b=0,c_{b+1} >0\},
\end{align} we can write \be \sum_{z^0 \in A} \pi^Q (z^0)  f_B(z^0)\ =\ \sum_{m=0}^3  \sum_{z^0 \in A} \pi^Q (z^0) f_{B_m} (z^0).\ee Note that $\cup_{m=0}^3 B_m = \{(0,b,b') \in {\cal L}: c_{b-1}  \ne 0 \mbox{ or } c_{b'} \ne 0\}$, and so this union does not necessarily contain all elements $(0,b,b') \in {\cal L}$. However, it does contain all such elements which are accessible from $A$ in one step (and more, whenever $B_3$ is not empty). \\

We now simplify the expression, beginning with the sum over $B_1$. It is important to observe that $B_1$ is the subset of states in $B$ accessible in one step  only  from $A$,  In addition, if $z^1 \in B_1$, then it immediately follows that $z^2=\dots = z^{k-1} = (0,1,1)$, and that allowed transitions to $(0,1,1)$ always have probability $\lambda_C^{-1}$.   As a result, we have that
\be f_{B_1} (z^0)\ =\ Q(z^0, z^1) \lambda_C^{-(k-2)} (1-\lambda_C^{-1}),\ee
and thus  \be \sum_{z^0 \in A} \pi^Q (z^0) f_{B_1}(z^0)\ =\ \pi^Q (B_1) \nu (k-1) .\ee

Next we consider the sum over $B_0$, namely $z^1 = (0,1,1)$. Clearly:
\begin{align}
\sum_{z^0 \in A} \pi^Q (z^0) f_{(0,1,1)} (z^0) &\ = \  \sum_{z^0\in {\cal L}} \pi^Q (z^0) f_{(0,1,1)} (z^0) - \sum_{z^0 \in B }\pi^Q (z^0) f_{(0,1,1)} (z^0).
\end{align}
Since  $(0,1,1)$ is accessible in one step either from $A$ or from states in $z \in B_0 \cup B_1\cup B_3$ and for all such $z$, $Q(z,(0,1,1)) = \lambda_C^{-1}$, it  follows that
 \be \sum_{z^0 \in A} \pi^Q (z^0) f_{(0,1,1)} (z^0)\ = \ \left ( \pi^Q ( (0,1,1) ) -\pi^Q(B_0 \cup B_1 \cup B_3 )\lambda_C^{-1} \right)  \nu (k-1).\ee
 Hence,
 \be   \sum_{z^0 \in A} \pi^Q (z^0) f_{B_0 \cup B_1} (z^0) \ =\ \pi^Q (B_0 \cup  B_1\cup B_3 ) (1-\lambda_C^{-1}) \nu (k-1) - \pi^Q (B_3) \nu (k-1).\ee


We now consider  $z^1 \in B_2$. Suppose then that  $z^0 \in A$ and $z^1 \in B_2$ and $Q(z^0,z^1)>0$. Since $z^1 = (0,b+1,b+2)$, it follows that $z^0=(c_b,b,b+1)$ and  $c_b >0$. Now if $c_{b+2} =0$, then the only allowed transition from $z^1$ is to $z^2 = (0,b+2,b+3)$. Let  $r=r(b)$ and $\rho =\rho(b)$ as defined in the statement of the theorem. Then  $z^j=(0,b+j,b+j+1)$ for all $j=1,\dots r$, and we conclude that $Q(z^j,z^{j+1}) =1$ for $j=0,\dots, r$. We continue according the the following two cases. \\
~\\
\noindent {\bf 1. $r>  k-1$.} In this case $Q^k (z^0,A)=0$.\\
\noindent{\bf 2. $r \le k-1$.} Then either
\begin{itemize}
\item $r=k-1$, in which case $Q^k(z^0,A) = Q( (0,b+r,b+r+1),A)=1-\rho$;   or
\item $1< r\le k-2$, in which case $z^{r+1} = (0,b+r+1,1)$ and $z^{r+l}= (0,1,1)$ for all $2\le l \le k-1-r$. In particular, since $Q ( (0,1,1) , A) = Q ( (0,b+r+1,1),A)= 1-\lambda_C^{-1}$, we have that
\be Q^k(z^0,A)\ =\ Q ( (0,b+r,b+r+1), (0,b+r+1,1))\nu (k-r-1).\ee
\end{itemize}

The only allowed transitions from  $(0,b+r,b+r+1)$ to $(x,b+r+1,1)$ are to $x=0,\dots,c_{b+r+1}-1$, all with equal transition probability.  Since there are exactly $c_{b+r+1}-\delta_L(b+r+1)$ possible values for $x$, exactly one of which is with $x=0$, letting $\rho (b) = \frac{1}{c_b - \delta_L (b)}$, we have
 \be Q^k (z^0,A)\ =\ \begin{cases} 0 & k < r (z)+1 \\  1- \rho (b+r+1) & k = r(z)+1 \\ \rho (b+r+1) \nu (k-r-1) & k> r(z)+1.\end{cases} 
  \ee

Summarizing the two cases, we conclude that
  \be \sum_{z^0 \in A} \pi^Q (z^0) f_{B_2} (z^0)\ =\ \sum_{z^1 \in B_2} h (z^1(2)-1,k).\ee

Next, when $z^0\in A$,  and $z^1\in B_3$, then $Q(z^0,z^1)=0$. Thus,  we have proved
\be \sum_{m=0}^3 \pi^Q (z^0) f_{B_m} (z^0)\ = \ \left( (1-\lambda_C^{-1})  \pi^Q (H_1) - \pi^Q (B_3) \right)  \nu (k-1)  + \sum_{z^0=(0,b+1,b+2) \in B_2} \pi^Q(z^0) h(b-1,k).\ee
Let $z' \in B_3$. Then there exists a unique $z^1=(0,b+1,b+2)  \in B_2$ such that  $z^1 = (0,b,b+1),~z^2 = (0,b+2,b+3),\dots, z^{r(b)} =(0,b+r(b),b+r(b)+1)$ and $z^{r(b)+1}=z'$.  Since $Q(z^k,z^{k+1})=1$ for $k=1,\dots,r(b)-1$, it easily follows that $\pi^Q (z') = \pi^Q (z^r) \rho (b)=\pi^Q (z^{r-1}) \rho (b)=\dots = \pi^Q (z^1) \rho(b)$. This shows that $\pi^Q(B_3) = \sum_{\{z^1 =(0,b+1,b+2) \in B_2\} }\pi^Q (z^1) \rho(b)$. Plugging this into the formula above, and noting that $H_1$ in the theorem is $B_0 \cup B_1\cup B_3$ and $H_2$ in the theorem is $B_2$, we obtain
\begin{eqnarray}\label{eq:place_holder}
\lim_{n\to\infty}   \frac 1n E^Q_z N_n(k)&\ =\ & \sum_{m=0}^2 \pi^Q (z^0) f_{B_m} (z^0)\nonumber\\
 &\ = \  & (1-\lambda_C^{-1})  \pi^Q (H_1)\nu(k-1) \nonumber\\ & & \ \ \   +\ \sum_{z^0=(0,b+1,b+2) \in H_2} \pi^Q ( z^0) \left ( h(b-1,k) -\rho(b) \nu(k-1))\right).
 \end{eqnarray}

We  turn to gaps of length $1$:
 \begin{align} \frac 1n E^Q_z N_n(1) &\ = \    \frac 1n E^Q_z \sum_{j=0}^{n-1} \ch_{\{Z_j(1) >0\}}\ch_{\{Z_{j+1} (1) >0\}}\nonumber\\
  &\ = \ \frac 1n \sum_{j=0}^{n-1} E_z^Q\ch_{\{Z_j(1) > 0\}}E_{Z_j}^Q \ch_{\{Z_1(1)> 0\}},
\end{align}
where in the second line we applied the Markov property. Let
\be A \ = \ \{(x,b,b')\in {\cal L}: x> 0\}.\ee
Ergodicity of $Z$ under $Q$ then gives
\begin{align}
\label{eq:nice1}
  \lim_{n\to\infty} \frac 1n E^Q_z N_n(1)\ =\ \sum_{z\in A } \pi^Q(z) Q(z,A) = \pi^Q (A) - \sum_{z \in A^c} \pi^Q (z) Q (z,A).
\end{align}

Given $z=(0,b,b') \in A^c$, exactly one of the following holds.
  \begin{itemize}
  \item $c_b=0$, $b'=b+1$, $c_{b+1}=0$, and then $Q(z,A)=0$.
  \item  $c_b=0$, $b'=b+1$, $c_{b+1}>0$. From the argument in the paragraph above \eqref{eq:place_holder}, and since $Q(z,A) = 1-Q(z,A^c)$ we obtain that
  \bea \sum_{\{z=(0,b,b+1)\in {\cal L}: c_b =0,c_{b+1} =1\}} \pi^Q (z) Q(z,A) &\  =\ & \pi^Q(B_2) - \pi^Q (B_3) \nonumber\\ &=& \sum_{z^0=(0,b+1,b+2) \in H_2} \pi^Q (z^0) (1- \rho(b)).\ \ \ \eea
  \item $c_b >0$ and then $b'=1$, equivalently, $z \in H_1$, in which case $Q(z,A) = 1-Q(z,(0,1,1))=1-\lambda_C^{-1}$.
  \end{itemize}
  Summarizing,
  \begin{align}
  \label{eq:nice2}
    \lim_{n\to\infty} \frac 1n E^Q_z N_n(1) = 1-\pi^Q_1(0) - (1-\lambda_C^{-1})\pi^Q_1(H_1)-\left ( \sum_{z^0=(0,b+1,b+2) \in H_2} \pi^Q (z^0) (1- \rho(b))\right).
  \end{align}

To finish the proof, we need to show that the results continue to hold when considering the measure $Q^n$ instead of $Q$. However, by the Markov property, the expectation under $Q^n$ of $N_n$, and $N_n(k)$ are equal to the expectations of corresponding additive functionals. Therefore  it follows from Theorem \ref{th:additive_functionals} that the expectations of $N_n(k)$ and $N_n$ under $Q_n$ are asymptotical equivalent to their expectations with respect to $Q_{\varphi_c}$. The theorem now follows.\end{proof}

\begin{proof}[Proof of Theorem \ref{th:weaklim_prob}]
We have
 \bea \{ |\hat \mu_n (A) - \mu_\infty (A) |>\epsilon \} & \ \subset \ & \cup_{k\in A} \{ \left | \hat \mu_n (k) - \mu_\infty (k) \right| > \epsilon\}\nonumber\\ &\ = \  & \cup_{k\in A} \{ \left | N_n (k) -\mu_\infty (k) N_n  \right|> \epsilon  N_n \}\cup \{N_n =0\}.\eea
 Since $Q_n (N_n=0) =Q_n (Z_0>0,Z_1=\dots =Z_n=0)\to 0$, we can ignore the event $\{N_n=0\}$. Now for every fixed $k\in A$, we have
\be \{ \left | N_n (k) -\mu_\infty (k) N_n  \right|> \epsilon  N_n \}\subset \{ |N_n(k) - \mu_\infty (k) E_{\pi^Q}^Q  N_n | > \epsilon/2 \} \cup \{ | N_n - E_{\pi^Q}^Q N_n| > \epsilon/2\}.\ee
Next observe that both $N_n$ and $N_n(k)$ are additive functionals for the process $Z^k = (Z^k_n:k\in\Z_+)$, where $Z^k_n =(Z_n,Z_{n+1},\dots,Z_{n+k})$, and so we can consider $N_n(k)$ and $N_n$ as additive functionals of $Z^k$.  Letting $ \varphi' _c(z^0,z^1,\dots,z^k)= \varphi_c (z^0)$, and $\tilde \varphi_c'(z^0,\dots,z^k)$, the distribution of $Z_0,\dots,Z_k$ under $Q_{\tilde\varphi_c}$, then if as in Definition \ref{def:additive} we define
\be Q'_{n,k} (A)\ =\ \frac{E_{\tilde \varphi_c'}^Q \left (\frac{\ch_{A}}{\varphi_c'(Z_n^k)}\right)}{E^Q_{\tilde \varphi_c'}\left ( \frac{1}{\varphi_c'(Z_n^k)}\right)},\ee
it follows that the restriction of $Q'_{n,k}$ to events generated by $Z_0,\dots,Z_n$ coincides with $Q_n$.  In particular,  the distribution of the additive functionals $N_n$ and $N_n(k)$ for $Z^k$ under $Q_{n,k'}$ coincides with their distribution under $Q_n$. From the variance estimate \eqref{eq:sec_mom} in Theorem  \ref{th:additive_functionals} applied to these additive functionals under $Q_{n,k}'$, we conclude that
 \be Q_n ( \{  |N_n(k) - \mu_\infty (k) E_{\pi^Q}^Q  N_n | > \epsilon/2 \}  ) =O(n^{-1})\mbox{ and }Q_n (\{ | N_n - E_{\pi^Q}^Q N_n| > \epsilon/2\})\ =\ O(n^{-1}).\ee
 Therefore if $A$ is finite, we obtain that
 \be \label{eq:weak_conv} \lim_{n\to\infty} Q_n (  |\hat \mu_n (A) - \mu_\infty (A) |>\epsilon ) \ =\ 0.\ee
 Now if $A$ is infinite,  letting $A_M = A \cap \{0,\dots,M\}$, we observe that
\begin{align}
\nonumber
 | \hat \mu_n (A) - \mu_\infty (A) |&\ = \  |\hat \mu_n (A_M) - \mu_\infty (A_M) | + \hat \mu_n (\{M+1,\dots\}) + \mu_\infty (\{M+1,\dots \})\nonumber\\
 &\ \le\  |\hat \mu_n (A_M) - \mu_\infty (A_M) | + \hat\mu (\{M+1,\dots\})+ \mu_\infty (\{M+1,\dots\}).
 \end{align}
Fix $\epsilon$, and let $M$ be such that $\mu_{\infty}(\{M+1,\dots\})< \epsilon$. Thus for $n$ large enough,
\be \{ | \hat \mu_n (A) - \mu_\infty (A) |>5\epsilon \}\subset \{ |\hat \mu_n (A_M) - \mu_\infty (A_M) |  > 2\epsilon\}\cup \{ \hat \mu_n (\{M+1,\dots\}) >2\epsilon\}.\ee
The measure of the first event on the right-hand side tends to $0$ as $n\to\infty$ by \eqref{eq:weak_conv}. As for the second event, it is equal to the event $\{\hat \mu_n (\{0,\dots,M\}) <1-2\epsilon\}$. However, since, again by \eqref{eq:weak_conv}  $\mu ( |\hat \mu_n (\{0,\dots,M\})- \mu_\infty (\{M+1,\dots\})|> \epsilon/2)$  tends to $0$, it follows that $Q_n ( \{ \hat \mu_n (\{0,\dots,M\}> 1-3\epsilon /2)$ tends to $1$. But this event is $\{ \hat \mu_n (\{M+1,\dots,\})<3 \epsilon /2\}$, and so $Q_n ( \hat \mu_n (\{M+1,\dots,\}) > 2\epsilon) $ tends to $0$ as well. The result now follows.\end{proof}

\subsection{Maximal Gap}

Next we consider the maximal gap $M_n$, defined as \begin{equation} M_n \ = \ \sup\{k\in \Z_+: N_n(k) >0\}.\end{equation}

Although we can prove the results at the same level of generality as in the previous section, we prefer to keep the expressions cleaner and simpler, and will assume throughout this section that $c_1,\dots,c_L >0$. \\

Our analysis is based on a renewal structure we now describe. We refer to the gaps of length $k\ge 2$  as ``long gaps", and denote the lengths of the long gaps, indexed by order of appearance, by $(R_j:j\in\N)$. Observe that any long gap is  followed by a possibly empty sequence of: gaps of  zero  length (summand repeated more than once, see first paragraph of Section \ref{sec:gap_dist}) and gaps of length $1$,  independent of $k$. This is then  followed again by an independent long gap. The number of the small gaps is bounded above by $(L-1)+\sum_i (c_i-1)=(\sum_i c_i)-1$, as  the first summand bounds the number of length  $1$,  and the second summand bounds the number of gaps of length  zero.   
Let $T_m$ denote the first time exactly $m$ long gaps are completed,   $ m(n) = \sup\{m : T_m \le n\}$. Observe that a long gap is completed whenever the digit zero is followed by a nonzero digit. Therefore
\begin{equation}m(n) \ = \ \sum_{j=0}^{n-1} \ch_{0}(Z_j(1))\ch_{\{Z_{j+1}(1)> 0\}}.\end{equation}
From the Markov property,
\begin{equation} E^Q m(n) \ = \ \sum_{j=0}^{n-1} E^Q \left ( \ch_{0} (Z_j(1))Q_{Z_j}(Z_1(1)> 0)\right),\end{equation}
Letting $A= \{z= (x,b,b') \in{\cal L}: x>0\}$, and repeating a similar computation as in the proof of the case $k=1$ in Theorem \ref{th:gap_dist}, it follows that
\begin{align}
 \lim_{n\to\infty} \frac{1}{n} E^Q m(n) \ &\ = \  \ \sum_{z\in A^c} \pi^Q(z)Q(z,A)\ = \pi^Q(A^c) - \sum_{z \in A} \pi^Q (z) Q(z,A) \nonumber\\
  &\ = \  (1-\pi^Q(0))- (1- \pi^Q(0))+(1-\lambda_C^{-1}) \pi^Q(0),\end{align}
where the last equality follows from \eqref{eq:nice1} and \eqref{eq:nice2}. Also, by the renewal theorem \cite[Theorem 2.4.6]{durrett}
 \begin{equation}
 \label{eq:renewal}
 \lim_{n\to\infty} \frac{m(n)}{n}  \ = \ \alpha,~Q\mbox{-a.s.},
 \end{equation}
 where $\alpha = 1 / E_{\rho}^Q T_1$ and $\rho$ is the uniform distribution on $c_1$ elements: $(x,1,1),~1<x<c_1$ and $(c_1,1,2)$.
The limit above also holds in $L^1(Q)$, as $m(n)\le n$. Consequently
\begin{equation}\alpha \ = \ \pi^Q_1(0)\left(1-\frac{1}{\lambda_C}\right).\end{equation}

To state our result we need to introduce some additional assumption. We say that a sequence $(n_k:k\in\N)$ of natural numbers tending to $\infty$ satisfies the {\it spacing condition} with respect to $\alpha$ and $q$ if
\begin{equation}\label{eq:spacingcondition}\liminf_{k\to\infty} \inf_{z\in \Z_+} \left| \frac{ \ln (n_k \alpha)}{\ln \frac1q}  -z\right|\ >\ 0.\end{equation}

Roughly speaking, this means that $n_k\alpha$ is eventually uniformly far from integer powers of $1/q$ in some  normalized sense.

\begin{theorem}
\label{th:maxgap}
Assume $c_1c_2\cdots c_L>0$. Then for every $k\in \Z$,
 \begin{equation} \lim_{n\to\infty} Q_n\left(M_n \ \le \ \left\lfloor \frac{\ln n \pi_1(0)(1- \frac{1}{\lambda_C})}{\ln \lambda_C}\right\rfloor +k\right) \ = \ e^{-\lambda_C^{-(k-2)}},\end{equation}
 when the limit is taken along any sequence satisfying the spacing condition \eqref{eq:spacingcondition} with respect to $\alpha = \pi_1(0)(1- \frac 1 {\lambda_C})$ and $q=\frac1{\lambda_C}$.
\end{theorem}

\begin{example}
 For the standard Zeckendorf decomposition, $\lambda_C = \phi$ and $\pi_1(0)= \frac{\phi+1}{\phi+2}$. This gives
   \begin{equation}\lim_{n\to\infty} Q_n\left(M_n \le \left\lfloor \frac{\ln n- \ln (\phi+2) }{\ln \phi }\right\rfloor +k\right) \ = \ e^{-\phi^{-(k-2)}}.\end{equation}
\end{example}

\begin{proof}[Proof of Theorem \ref{th:maxgap}]
To prove the theorem, we need to recall some facts on the maximum of negative geometric random variables. Let $\G$ be a negative geometric random variable with parameter $p\in (0,1)$. That is, for $k \in \Z_+$, $P( \G \ge k) = q^k$ where $q=1-p$.
Let $G$ be negative geometric with parameter $p$. That is, $\G$ takes values in $\Z_+$, and $P(\G\ge k) = q^k$, where $q=1-p$. We denote this distribution by ${\rm Geom}^-(p)$. Let $(\G_k:k\in\N)$ be IID ${\rm Geom}^-(p)$-distributed random variables, and let $M^\G_m = \max_{k\le m} \G_k$. Then $ P( M^\G_m \le j) = (1-q^j)^m$. For each $m\in \N$, let  $\delta_m$ be chosen so that $\frac{\ln m \delta_m}{\ln 1/q} = \left\lfloor \frac{\ln m}{\ln 1/q}\right\rfloor$. Observe then that $\delta_m \in (q,1]$. From this we obtain that for any $k\in\Z$,
 \begin{equation}
 \label{eq:max_geom}
 P\left(M_m^\G \le \left\lfloor \frac{\ln m}{\ln 1/q}\right\rfloor + k\right) \ = \ \left(1- \frac{q^k}{m\delta_m}\right)^m\underset{m\to\infty}{\to} e^{-q^k}.
 \end{equation}

We return to the proof. Fix some sequence satisfying the spacing condition. Abusing notation, we will refer to a generic element in the sequence as $n$. Observe that if we choose $\G_j=R_j-2$, then $(\G_j:j\in \N)$ is an IID sequence of ${\rm Geom}^-(p)$ random variables with $p=1-\lambda_C^{-1}$. In particular, for every $m$
\begin{equation} M_{T_m} \ = \ M^\G_m +2.\end{equation}

Clearly $T_{m(n)} \le n$, but also by  the law of large numbers and \eqref{eq:renewal}
  \begin{equation} \frac{ T_{m(n)}}{ n }\ = \ \frac{  T_{m(n)}}{m(n)} \times  \frac{ m(n) }{n} \underset{n\to\infty}{\to}  1,~Q\mbox{-a.s.}\end{equation}

From \eqref{eq:renewal} we can find  $\epsilon_n>0$ with $\lim_{n\to\infty} \epsilon_n = 0$ and  satisfying
\begin{equation}Q\left( \frac{m(n)}{n} \in [1-\epsilon_n,1+\epsilon_n]\alpha\right) \  \underset{n\to\infty} {\to}\ 1.\end{equation}
Observe then that
\begin{eqnarray} Q\left( M_n \le \left\lfloor \frac{\ln n \alpha}{\ln \frac1q}\right\rfloor + k\right) &\ \ge\ &
Q\left(M_n \le  \left\lfloor \frac{\ln n \alpha}{\ln \frac1q}\right\rfloor + k, 0< m(n) \le (1+\epsilon_n) n \alpha\right)\nonumber\\
&\ge &
 Q\left(M^\G_{\left\lfloor(1+\epsilon_n)n \alpha \right\rfloor} \le  \left\lfloor \frac{\ln n \alpha}{\ln \frac1q}\right\rfloor + k-2\right) - Q (m(n) \nonumber\\ &>&  (1+\epsilon_n) n \alpha)-Q(m(n)=0).
 \end{eqnarray}
The last two terms on the righthand side tend to $0$. In addition, since  $\ln (n(1+\epsilon_n) \alpha) - \ln (n \alpha)\underset{n\to\infty} {\to} 0$, it follows from the spacing condition that for all $n$ large enough,$ \left\lfloor \frac{ \left\lfloor(1+\epsilon_n)n \alpha \right\rfloor}{\ln \frac 1q}\right\rfloor= \left\lfloor \frac{\ln n \alpha}{\ln \frac1q}\right\rfloor$. It then follows from \eqref{eq:max_geom} that \begin{equation}\liminf_{n\to\infty}  Q\left(M_n \le \left\lfloor \frac{\ln n \alpha}{\ln \frac1q}\right\rfloor + k\right)\ \ge\ e^{-q^{k-2}}.\end{equation}

We turn to the upper bound.
 \begin{eqnarray} Q\left( M_n \le \left\lfloor \frac{\ln n \alpha}{\ln \frac1q}\right\rfloor + k\right) &\ \le\ & Q\left(M_n \le  \left\lfloor \frac{\ln n \alpha}{\ln \frac1q}\right\rfloor + k-2,  m(n) \ge (1-\epsilon_n) n \alpha\right)\nonumber\\ & & \ \ \  + \ Q\left(m(n) < (1-\epsilon_n)n \alpha\right)\nonumber\\
 & \le & Q\left( M^\G_{\lceil (1-\epsilon_n)n \alpha\rceil} \le \left\lfloor \frac{\ln n \alpha}{\ln \frac1q}\right\rfloor + k\right) +o(1).
 \end{eqnarray}
 The same argument as before shows that for $n$ large enough, $\left\lfloor
 \frac{\ln \lceil(1-\epsilon_n)n \alpha \rceil }{\ln \frac 1q}\right\rfloor =\left\lfloor\frac{\ln n \alpha}{\ln \frac1q}\right\rfloor$, and so
 \begin{equation} \limsup_{n\to\infty} Q\left( M_n \le \left\lfloor \frac{\ln n \alpha}{\ln \frac1q}\right\rfloor + k\right)\  \le \ e^{-q^{k-2}}.\end{equation}

 Summarizing,
 \begin{equation}
\label{eq:QM_lim}
\lim_{n\to\infty}  Q\left(M_n \le \left\lfloor \frac{\ln n \alpha}{\ln \frac1q}\right\rfloor + k\right)\ =\ e^{-q^{k-2}}.
\end{equation}

It remains to convert the result to $Q_n$. Let $A_n =\{M_n \ge \left\lfloor \ln \ln n\right\rfloor \}$. Then $Q(A_n) \underset{n\to\infty}{\to} 1$. Let $b_n = \left\lfloor \ln \ln n \right\rfloor$.  Then as $n-b_n = n (1 +o(1))$, we conclude that the sequence $n- b_n$ also satisfies the spacing condition. Furthermore, for sufficiently large $n$, $\left\lfloor \frac{\ln (n -b_n) \alpha } {\ln 1/q} \right\rfloor = \left\lfloor \frac{\ln n  \alpha } {\ln 1/q} \right\rfloor$. Thus, from \eqref{eq:QM_lim}
\begin{equation} \lim_{n\to\infty} Q\left( M_{n- b_n}  \le \left\lfloor \frac{\ln n \alpha}{\ln \frac 1q} \right\rfloor + k\right) \ = \ e^{-q^{k-2}}.\end{equation}
 Letting  $B_n = \{M_{n- b_n} \le  \left\lfloor \frac{\ln n \alpha}{\ln \frac 1q} \right\rfloor + k\}$, it follows from the Markov property and the ergodicity of $Z$ that
  \begin{align} E^Q \left ( \ch_{B_n} \frac{1}{\varphi_c(Z_n)}\right) &\ = \   E^Q \left ( \ch_{B_n} E_{X_{n-b_n}} \frac{1}{\varphi_c(X_{b_n})} \right) \nonumber\\
\label{eq:ergo_pass}
  &\ = \ E^Q \left ( \ch_{B_n} E_{\pi^Q} \frac{1}{\varphi_c}\right) + o(1) \ =\  Q(B_n)+o(1).
  \end{align}
  Now
 \begin{align} Q\left( M_n\le \left\lfloor \frac{\ln n \alpha}{\ln \frac 1q} \right\rfloor + k, \frac{1}{\varphi_c(X_n)}\right) &\ \le\
 Q\left( \ch_{B_n}, E^Q_{X_{n-b_n}} \frac{1}{\varphi_c(X_{b_n})}\right) \nonumber\\
 & \ = \ Q(B_n) E_{\pi^Q} \frac{1}{\varphi_c}+ o(1),
 \end{align}
and so
 \begin{equation}\limsup_{n\to\infty} Q_n\left(M_n \le \left\lfloor \frac{ \ln n\alpha}{\ln 1/q}\right\rfloor+k\right) \ \le\  e^{-q^{k-2}}.\end{equation}

We turn to the lower bound. Observe that $M_n > M_{n-b_n}$ only if one of the last $b_n+1$ long gaps among the first $m(n)$ is maximal. Fix $c>0$, then for all $n$ large enough, depending on $c$ and on  the event $\{M_n > c \ln n\}$, those maximal gap among the last $b_n+1$ must begin before $n-b_n$ (because otherwise it will have length at most $b_n < c \ln n$) and end after $n-b_n$ (otherwise already included in $M_{n-b_n}$).  That is,
\begin{equation}\{M_n > M_{n-b_n} \} \cap \{M_n > c \ln n\}\subset \{ \max_{j=1,\dots, m(n-b_n)+1} \G_j \ = \ \G_{m(n-b_n)+1}\}.\end{equation}
Denote the event on the right-hand side by $C_n$. We have that
\begin{align}  Q(C_n) & \ \le\ Q(C_n, m(n) \in (  (1-\epsilon) n \alpha, (1+\epsilon)n \alpha)) + o(1)\nonumber\\
&\ \le\  2 \epsilon n \alpha \times \frac{1}{(1-\epsilon)n \alpha}+o(1)\underset{n\to\infty} { \to} \frac{2\epsilon}{1-\epsilon}.
\end{align}
Since $\epsilon$ is arbitrary, we conclude that $Q(C_n) \underset{n\to\infty}{\to}0$.
Hence
\begin{align}
E^Q\left( M_n > \left\lfloor\frac{\ln n \alpha}{\ln \frac 1q}\right\rfloor + k,\frac{1}{\varphi_c(Z_n)}\right) &\ \le\ Q\left(M_{n-b_n} > \left\lfloor \frac{  \ln n \alpha}{\ln \frac 1q}\right\rfloor + k, C_n^c,\frac{1}{\varphi_c(X_n)}\right)+Q(C_n)\nonumber\\
&\ \le\ Q\left(M_{n-b_n} > \left\lfloor \frac{  \ln (n-b_n) \alpha}{\ln \frac 1q}\right\rfloor + k,\frac{1}{\varphi_c(Z_n)}\right)+o(1).
\end{align}
The remainder of the proof is identical to the argument presented in \eqref{eq:ergo_pass}, with the obvious changes. This gives the lower bound
\begin{equation} \liminf_{n\to\infty}Q^n\left(M_n \le \frac{\left\lfloor  \ln n \alpha\right\rfloor}{\ln \frac 1q} + k\right)\ \ge\  e^{-q^{k-2}},\end{equation}
thus completing the proof.\end{proof}

\section*{Appendix: Generalization to initial segments}\label{sec:geninitialseg}
Although our approach is most natural for intervals of the form $[G_n,G_{n+1})$, most of the results can be easily extended to the general case where we consider the interval $[1,N)$. We will now briefly show how this can be done. For every $N\in\N$ there exists a unique $n=n(N)$ such that $N \in [G_{n+1},G_{n+2})$. Denote the uniform measure on $[1,N)$ by $W_N$. Then it follows from  Theorem \ref{th:structure} that
\be W_N(A)\ =\ \sum_{j=0}^{n(N)-1} \alpha_j Q^j (A)+ \alpha_n Q^n ( A |  [G_{n+1} , N)),\ee
where $G_0=0$, $\alpha_j =  (\min (G_{j+2},N)-G_{j+1}) / N$,  and  for simplicity we consider  $Q^j$ as a probability measure on $\N$ which gives zero mass to elements outside the interval $[G_{j+1},G_{j+2})$. \\

Consider now a function $F:\N \to [0,1]$, and assume that  $\lim_{j\to\infty} E^{Q_j} F = c(F)$. We will make more assumptions on $F$ later. We will basically require $F$ not to depend too much on its first and last digits. \\

Then we can write
\be \label{eq:eval_0N} E^{W_N} F\ =\ \sum_{j=1}^{N(n)-1} \alpha_j E^{Q^j }F +  \alpha_n E^{Q^n }( F   | [G_{n+1} , N))\ =\ (I)+(II).\ee

It is easy to see that along sequences satisfying either  $\alpha_n \to 0$ or $\alpha_n \to 1$, the righthand side converges to $c(F)$. However, when this is not the case, then the term (II) may be oscillatory. However, if we can show that  $E^{Q^n }( F   | [G_{n+1} , N))\to c(F)$, then it follows that $E^{W_N} F \to c(F)$. The idea is very much in the  spirit of Proposition \ref{pr:asymp}. This cannot hold for all $F$, so we need to restrict our discussion to those $F$ not affected much by first or last digits. \\

In order to do this we make some assumptions of $F$ so that the oscillations will asymptotically vanish. Denote the length of the decomposition of $x$ by $|x|$. Suppose that for each $x$ large enough, there exists $n_x$ such that $n_x,|x|-n_x\to \infty$, and if $A_x $ denotes all numbers with length $|x|$ whose decomposition differs from that of $x$ only in the first $n_x$ or last  $n_x$ digits, we will assume
\be \label{eq:not_too_much} \lim_{x\to\infty} \sup_{x' \in A_x} |F(x')-F(x)|\ =\ 0.\ee

An example of such a function is any additive functional  $S$, divided by the length of the decomposition $D$ (a random variable we localized to numbers with decompositions of fixed length in previous sections).  Another example is $e^{i \theta (S-c)/\sqrt{D}}$ for some $c$. We note that we can make weaker assumptions on $F$ for the argument to work. Before presenting the argument, we state the result:

\begin{prop}
Suppose that $F:\N \to [0,1]$ satisfies $\lim_{j\to\infty} E^{Q_j} F = c(F)$. If  \eqref{eq:not_too_much} holds, then $\lim_{N\to\infty} E^{W_N} F=c(F)$.
\end{prop}


\begin{proof}
Assume then that we have a sequence  $N_1 < N_2 < \cdots$ with $n_1(N_1) \le n_2 (N_2) \le \cdots$.  Without loss of generality, we may assume  that that $n_j < n_{j+1}$ and $\inf  \alpha_{n_j} > \rho\in (0,1)$. That is,
 \be \label{eq:important} (N_j - G_{n_j+1})/ N_j\  >\  \rho,\ee
which in turn implies  $N_j  > (1+c_2) G_{n_j+1}$. This along with the exponential growth of $(G_n)$, guarantee that for any $\epsilon>0$, there exists some $K\in \N$ and $\tilde N_j \in [G_{n_j+1},N_j)$, such that
\begin{enumerate}
\item The first $K$ digits of $\tilde N_j$ coincide with those of $N_j$.
\item All other digits of $\tilde N_j$ are zero.
\item $\tilde N_j / N_j \ge (1-\epsilon)$.
\end{enumerate}
In other words, the fact that $N_j$ is at least a certain fixed multiple (depending only on $(N_j)$ ) of $G_{n_j+1}$ means that the first digits may have some constraints, but not the last (because they cannot contribute much to the sum). This allows us to ``round" down $N_j$ to $\tilde N_j$, a near number for which the condition of being in the interval $[G_{n_j+1},\tilde N_j)$ is determined only by the first $K$  digits.\\
Now we repeat the  argument from  Proposition \ref{pr:asymp} which allows to separate the first $K<n_x$ and the last $n_x$ digits from the rest. This gives
\be \lim_{j\to\infty} E^{Q_{n_j}} \left ( F (X) | X< \tilde N_j\right)\ = \ c(F).\ee

The last step is to recover (II) for $N_j$ from the corresponding expression for $\tilde N_j$.  We have
\be E^{Q_{n_j}} \left(F(X), X < N_j\right)\  =\ E^{Q_{n_j}} \left ( F(X), X < \tilde N_j\right) + E^{Q_{n_j}} \left(F(X),   \tilde N_j\le X < N_j \right).\ee

By condition (3) in the choice of $\tilde N_j$, the  absolute value of second summand on the righthand side is bounded above by $( N_j - \tilde N_j)/(G_{n_j+2}-G_{n_j+1}) \le  \epsilon \frac{ N_j }{G_{N_{j+2}} - G_{N_j+1}}$. This implies
\be E^{Q_{n_j}} \left ( F(X)| X < N_j\right)\ =\ E^{Q_{n_j}} \left ( F(X)| X < \tilde N_j\right) \frac{ |\tilde N_j -G_{n_{j+1}}|}{|N_j - G_{n_{j+1}}|}+ \epsilon \frac{N_j}{| N_j -G_{n_{j+1}}|}O(1).\ee
Thus,
\begin{align}
\nonumber
\left |E^{Q_{n_j}} \left ( F(X)| X < N_j\right) - E^{Q_{n_j}} \left ( F(X)| X < \tilde N_j\right)\right | &\ =\ \frac{ |N_j - \tilde N_j |}{N_j -G_{n_j+1}} O(1) + \epsilon \frac{N_j}{| N_j -G_{n_{j+1}}|}O(1) \\
&\ =\ \epsilon \frac{N_j} {N_j - G_{n_j+1}}O(1)=\epsilon O(1),
\end{align}
the first equality on the second line is from condition (3) in the choice of $\tilde N_j$, and the second equality there follows from \eqref{eq:important}. Therefore $  \left  |E^{Q_{n_j}} \left ( F(X)| X < N_j\right) - c(F)\right  | =O(\epsilon)$, and it then follows from \eqref{eq:eval_0N} that  $\limsup_j  | E^{W_{N_j}} F(X) -c (F)| = \epsilon O(1) $, completing the proof.
\end{proof}


\ \\

\end{document}